\documentclass[12pt]{amsart}
\usepackage{amsmath}
\usepackage{amscd}
\usepackage{amssymb}
\usepackage{amsfonts}
\usepackage{graphicx}
\usepackage{wrapfig}
\usepackage{amsthm}
\usepackage{euscript}
\usepackage[utf8]{inputenc}
\usepackage[english]{babel}
\usepackage{pgf,tikz}
\usepackage[colorlinks, citecolor=blue]{hyperref}

\newcommand{\de}{\partial}

\newcommand{\Ric}{\mathrm{Ric}}
\newcommand{\ov}[1]{\overline{#1}}

\newcommand{\vp}{\varphi}

\newcommand{\diam}{\mathrm{diam}}

\newcommand{\p}{\partial}

\renewcommand{\leq}{\leqslant}
\renewcommand{\geq}{\geqslant}
\renewcommand{\le}{\leqslant}
\renewcommand{\ge}{\geqslant}

\newcommand{\be}{\begin{equation}}
\newcommand{\ee}{\end{equation}}

\newcommand{\R}{\mathbb{R}}

\newcommand{\Hess}{\mathrm{Hess}}

\newcommand{\D}{\nabla}

\def\e{\varepsilon}

\def\O{\Omega}
\def\S{{\mathbb S}}
\def\<{\langle}
\def\>{\rangle}

\newtheorem{theorem}{Theorem}[section]
\newtheorem{lemma}[theorem]{Lemma}
\newtheorem{corollary}[theorem]{Corollary}

\theoremstyle{definition}

\theoremstyle{remark}
\newtheorem{remark}[theorem]{Remark}

\numberwithin{equation}{section}

\usepackage[utf8]{inputenc}
\usepackage{tikz}
\usetikzlibrary{patterns}
\pgfdeclarepatternformonly{my dots}{\pgfqpoint{-1pt}{-1pt}}{\pgfqpoint{2.5pt}{2.5pt}}{\pgfqpoint{6pt}{6pt}}%
{
	\pgfpathcircle{\pgfqpoint{0pt}{0pt}}{.5pt}
	\pgfpathcircle{\pgfqpoint{7pt}{7pt}}{.5pt}
	\pgfusepath{fill}
}
\begin{document}

\title[Spectral comparison and splitting theorems]
{Spectral comparison and splitting theorems for the infinity-Bakry-Emery Ricci curvature}

\author{Jia-Yong Wu}
\address{Department of Mathematics and Newtouch Center for Mathematics, Shanghai University, Shanghai 200444, China}
\email{wujiayong@shu.edu.cn}

\dedicatory{}
\date{\today}

\begin{abstract}
In this paper, we prove the diameter comparison, the global weighted volume comparison and
the splitting theorem in weighted manifolds when the infinity-Bakry-Emery
Ricci curvature has a lower bound in the spectrum sense. Our results extend
Antonelli-Xu's spectral Bonnet-Myers and Bishop-Gromov theorems, and
Antonelli-Pozzetta-Xu's spectral splitting theorem to weighted manifolds.
Our results are also some supplements of Chu-Hao's spectral diameter and
global volume comparisons, and Yeung's spectral splitting theorem in weighted
manifolds.
\end{abstract}
\maketitle

\section{Introduction}\label{Int1}
In \cite{AX17} Antonelli-Xu proved the diameter and global volume comparisons
in manifolds when the Ricci curvature has a positive lower bound in the spectrum
sense. Their results generalize the classical Bonnet-Myers and Bishop-Gromov
volume comparisons. Later, Chu-Hao \cite{CH} extended Antonelli-Xu's results
to weighted manifolds. To state their results, we need to introduce some concepts.
Recall that an $n$-dimensional weighted manifold, denoted by $(M,g,e^{-f}dv)$,
is an $n$-dimensional Riemannian manifold $(M,g)$ together with a weighted
measure $e^{-f}dv$ for some $f\in C^\infty(M)$, where $dv$ is the Riemannian
volume element on $M$. On $(M,g,e^{-f}dv)$, for a finite constant $N$, the
$N$-Bakry-Emery Ricci curvature introduced by Bakry-Emery \cite{BE} is defined by
\[
\Ric_f^N:=\Ric+\Hess\,f-\frac{df\otimes df}{N},
\]
where $\Ric$ is the Ricci curvature of $(M,g)$ and $\Hess$ is the Hessian
with respect to $g$. At a point $x\in M$, the smallest eigenvalue of the
$N$-Bakry-Emery Ricci curvature is defined by
\begin{equation}\label{RicfNsp}
\Ric_f^N(x):=\inf\left\{\Ric_f^N(v,v)\,\, |\,\,v\in T_xM,\, |v|=1\right\},
\end{equation}
which is locally Lipschitz on $M$. On $(M,g,e^{-f}dv)$, the $f$-Laplacian
$\Delta_f$ is defined by
\[
\Delta_f u:=\Delta u-g(\D f,\D u)
\]
for $u\in C^2(M)$, which is self-adjoint with respect to the weighted measure $e^{-f}dv$.

On $(M,g,e^{-f}dv)$, for constants $\gamma\ge 0$ and $\lambda\in\R$,
we say that
\[
\lambda_1(-\gamma\Delta_f+\Ric^N_f)\ge\lambda
\]
if any of the two following equivalent conditions holds:
\begin{enumerate}
\item [(1)]
$\int_M\left(\gamma|\D\vp|^2+\Ric_f^N\vp^2\right)e^{-f}dv\ge\lambda\int_M \vp^2e^{-f}dv$
for any $\vp\in C_c^1(M)$,
\item [(2)] there exist a constant $\alpha\in (0,1)$ and a positive function
$u\in C^{2,\alpha}(M)$ such that $u\,\Ric^N_f-\gamma \Delta_f u\ge\lambda u$ on $M$.
\end{enumerate}

When $N=\infty$, the $N$-Bakry-Emery Ricci curvature becomes the $\infty$-Bakry-Emery
Ricci curvature (also called the infinity-Bakry-Emery Ricci curvature)
\[
\Ric_f=\Ric+\Hess\,f.
\]
$\Ric_f$ is linked with the $f$-Laplacian via the generalized Bochner formula
\[
\Delta_f|\D u|^2=2|\Hess\,u|^2+2g(\D u,\D\Delta_f u)+2\Ric_f(\D u, \D u)
\]
for $u\in C^\infty(M)$, which plays a basic role in the comparison geometry
of weighted manifolds (see e.g. \cite{WW09}). $\Ric_f$ is also related to
the following gradient Ricci soliton
\[
\Ric_f=\lambda\, g,
\]
where $\lambda$ is a real constant, and it plays an important role in studying
the singularities of the Ricci flow \cite{Ham}. A gradient Ricci soliton is
called shrinking, steady, or expanding, if $\lambda$ is positive, zero, or
negative, respectively.

 Analogous to \eqref{RicfNsp},
one defines the smallest eigenvalue of $\Ric_f$ at $x\in M$ by
\[
\Ric_f(x):= \inf\left\{\Ric_f(v,v)\,\, |\,\, v\in T_xM,\,  |v|=1\right\}.
\]
For constants $\gamma\ge 0$ and $\lambda\in \R$, we say that
\[
\lambda_1(-\gamma\Delta_f+\Ric_f)\ge\lambda
\]
on $(M,g,e^{-f}dv)$ if any of the two following equivalent conditions holds:
\begin{enumerate}
\item [(1)]
$\int_M\left(\gamma|\D\vp|^2+\Ric_f\vp^2\right)e^{-f}dv\ge\lambda\int_M \vp^2e^{-f}dv$
for any $\vp\in C_c^1(M)$,
\item [(2)] there exist a constant $\alpha\in (0,1)$ and a positive function
$u\in C^{2,\alpha}(M)$ such that $u\,\Ric_f-\gamma \Delta_f u\ge \lambda u$
on $M$.
\end{enumerate}

Chu-Hao \cite{CH} generalized  Antonelli-Xu's spectral diameter and global
volume comparisons \cite{AX17} to weighted manifolds when the $N$-Bakry-Emery
Ricci curvature has a positive lower bound in the spectrum sense. They showed
that if there exist a constant $\lambda>0$ and a smooth function $u>0$ on an
$n$-dimensional ($n\ge 3$) weighted manifold $(M,g,e^{-f}dv)$ with
$0<\inf_Mu\le\sup_M u<\infty$, such that
\[
u\,\Ric_f^N-\gamma\Delta_fu\ge(n-1)\lambda u
\]
for two constants $N$ and $\gamma$ satisfying
$N\in(-\infty,-(n-1))\cup(0,+\infty)$ and $0\le\gamma\le\frac{N+n-1}{N+n-2}$,
and if $F:=\|f\|_{C^0(M)}<\infty$,
then one has the diameter comparison
\[
\diam(M,g) \leq
\begin{cases}
\left(\frac{\sup_Mu}{\inf_Mu}\right)^{\frac{N+n-3}{N+n-1}\gamma}
\cdot\sqrt{\frac{N+n-1}{n-1}} \cdot \frac{\pi}{\sqrt{\lambda}} & \mbox{\text{when $N>0$}}, \\[2mm]
\left(\frac{\sup_Mu}{\inf_Mu}\right)^{\frac{n-3}{n-1}\gamma}
\cdot e^{\frac{2}{n-1}F} \cdot \frac{\pi}{\sqrt{\lambda}} & \mbox{\text{when $N<-(n-1)$}},
\end{cases}
\]
and the global weighted volume comparison
\[
|M|_f\le e^{\frac{(n+1)(3n-1)}{n(n-1)}F}\cdot\lambda^{-\frac n2}|\S^n|,
\]
where $|M|_f:=\int_M e^{-f}dv$ is the weighted volume of the manifold $M$ and
$|\S^n|$ is the Riemannian volume of the unit sphere $\S^n$ in $\R^{n+1}$.
When $f=0$, the above weighted notations return to Riemannian cases and Chu-Hao's
results recover Antonelli-Xu's results \cite{AX17} by taking $F=0$ and $N=0$.

Notice that Chu-Hao's results require $N<\infty$. One may naturally ask if there exists
similar Chu-Hao's results in weighted manifolds when $\Ric_f$ has a positive
lower bound in the spectrum sense. In this paper we give a positive answer.
\begin{theorem}\label{result}
Let $(M,g,e^{-f}dv)$ be a complete $n$-dimensional ($n\geq3$) weighted manifold
and $\gamma$ be a constant such that $0\le \gamma\le 1$. Assume that there exists
a positive function $u\in C^{\infty}(M)$ such that
\[
u\,\Ric_f-\gamma\Delta_fu \geq (n-1)\lambda u
\]
for some constant $\lambda>0$. If $0<\inf_Mu\le\sup_Mu<\infty$ and
$F=\|f\|_{C^0(M)}<\infty$, then the following results hold:
\begin{enumerate}
\item Diameter comparison:
\[
\diam(M,g)\le\left(\frac{\sup_Mu}{\inf_Mu}\right)^{\frac{n-3}{n-1}\gamma}
\cdot e^{\frac{f_{max}-f_{min}}{n-1}}\cdot\frac{\pi}{\sqrt{\lambda}}.
\]
\item Global weighted volume comparison:
\[
|M|_f\le\exp\left\{\frac{n^2+2n-1}{n(n-1)}f_{max}-\frac{2n}{n-1}f_{min}\right\}
\cdot\lambda^{-\frac n2}|\S^n|.
\]
\end{enumerate}
Here $f_{max}:=\sup_{x\in M} f(x)$ and $f_{min}:=\inf_{x\in M} f(x)$.
\end{theorem}

\begin{remark}
(i) Theorem \ref{result} is not completely covered by Chu-Hao's results \cite{CH}. Indeed,
when $N>0$, Chu-Hao's spectrum condition $u\,\Ric_f^N-\gamma\Delta_fu\ge(n-1)\lambda u$
implies our spectrum condition $u\,\Ric_f-\gamma\Delta_fu \ge (n-1)\lambda u$; when
$-\infty<N<-(n-1)$, Chu-Hao's parameter range $0\le\gamma\le\frac{N+n-1}{N+n-2}$ is
narrower than our parameter range $0\le \gamma\le 1$. So our results are some
supplements of Chu-Hao's results in weighted manifolds. Besides, the dependence
on $f$ in our estimates is more precise than Chu-Hao's results \cite{CH}, which
may be useful for further understanding the geometry and topology of non-trivial
weighted manifolds.

(ii) When $f$ is constant, Theorem \ref{result} reduces to Antonelli-Xu's
results \cite{AX17} but our parameter range $0\le\gamma\le1$ is narrower than
their parameter range $0\le\gamma\le\frac{n-1}{n-2}$.

(iii) We indeed prove a more general estimate than Theorem \ref{result}(1) (see
Theorem \ref{resultk} in Section \ref{diacomp} below).
\end{remark}

Theorem \ref{result} is suitable to the closed shrinking gradient Ricci soliton.
In fact, we take $u\equiv 1$ and get
\begin{corollary}\label{resultc}
Let $(M,g,e^{-f}dv)$ be a closed $n$-dimensional ($n\geq3$)
shrinking gradient Ricci soliton satisfying
\[
\Ric_f=(n-1)\lambda
\]
for some constant $\lambda>0$. Then the following results hold:
\begin{enumerate}
\item [(1)] Diameter comparison:
\[
\diam(M,g)\le e^{\frac{f_{max}-f_{min}}{n-1}}\cdot\frac{\pi}{\sqrt{\lambda}}.
\]
\item [(2)] Weighted volume comparison:
\[
|M|_f\le\exp\left\{\frac{n^2+2n-1}{n(n-1)}f_{max}-\frac{2n}{n-1}f_{min}\right\}
\cdot\lambda^{-\frac n2}|\S^n|.
\]
\end{enumerate}
Here $f_{max}:=\sup_{x\in M} f(x)$ and $f_{min}:=\inf_{x\in M} f(x)$.
\end{corollary}

In Theorem \ref{result}, the proof idea of diameter comparison originates from the
$\mu$-bubble method in \cite{Gromov19} and its devolvements in \cite{Xu}; the
proof of global volume comparison originates from the idea of isoperimetric
comparison \cite{Bray97} and further developments in \cite{BP,Ba,BGLZ19,CLMS24}.
In this paper, we provide a brief proof of Theorem \ref{result} by modifying
Chu-Hao's argument \cite{CH}. For technique details, the reader may
refer to \cite{AX17} and references therein.

On the other hand, Antonelli-Pozzetta-Xu \cite{APX24} generalized the classical
Cheeger-Gromoll splitting theorem to the case when the Ricci curvature is
nonnegative in the spectrum sense. Catino-Mari-Mastrolia-Roncoroni \cite{CMMR24}
independently proved the same result using a different method. Hong-Wang \cite{HW}
established a spectral splitting theorem for manifolds with boundary by the second
variation of the weighted length functional and the curve-capture argument.
Recently, Yeung \cite{Yeu} extended the spectrum splitting theorem to weighted
manifolds when the $N$-Bakry-Emery Ricci curvature is nonnegative in the spectrum
sense. His result states that if an $n$-dimensional complete noncompact weighted
manifold $(M,g,e^{-f}dv)$ has at least two ends and $f$ is bounded, and if
there exist two constants $N\in (0,\infty)$ and $\gamma<\left[\tfrac{1}{(n-1)(1+\frac{n-1}{N})}+\tfrac{n-1}{4}\right]^{-1}$
such that
\begin{equation}\label{specNcon}
\lambda_1\left(-\gamma \Delta_f+\Ric^N_f\right)\ge 0,
\end{equation}
then $M$ is isometric to $\R\times X$ for some manifold $X$ with $(\Ric_X)^N_f\ge 0$.
When $f$ is constant,  if $N\to 0+$, then Yeung's result recovers
Antonelli-Pozzetta-Xu's splitting theorem \cite{APX24}.

Notice that Yeung's theorem requires $0<N<\infty$. One may naturally ask
if there exists a similar spectrum splitting theorem for $N=\infty$. In
this paper, we can generalize his result to weighted manifolds when
$\Ric_f$ is nonnegative in the spectrum sense.
\begin{theorem}\label{result2}
Let $(M,g,e^{-f}dv)$ be an $n$-dimensional complete noncompact weighted manifold
with at least two ends. Assume that $f$ is bounded on $M$. If there exists a constant
$\gamma<\left(\frac{1}{n-1}+\frac{n-1}{4}\right)^{-1}$ such that
\begin{equation}\label{speccond}
\lambda_1\left(-\gamma \Delta_f+\Ric_f\right)\ge0,
\end{equation}
then $\Ric_f\ge0$ on $M$. In particular, $(M,g,e^{-f}dv)$ is isometric to
$(\R\times X, dt^2+g_X)$ for some compact manifold $X$ with $(\Ric_X)_f\ge 0$.
\end{theorem}

\begin{remark}
(i) Wei and Wylie (see Example 2.2 in \cite{WW09}) pointed out that the
upper bound of $f$ is necessary in Theorem \ref{result2} with $\gamma=0$.
Indeed for any fixed point $p$ in the hyperbolic space ${\mathbb H}^n$,
let $f(x)=(n-1)d^2(p,x)$, where $d(p,x)$ denotes the distance from $p$ to
point $x\in{\mathbb H}^n$. Then $\mathrm{Ric}_f\ge(n-1)$. But the
Cheeger-Gromoll type splitting theorem does not hold on ${\mathbb H}^n$.

(ii) The lower bound of $f$ is used in two steps in the proof of Theorem
\ref{result2}. One is the auxiliary density estimate for the unequally
warped $\mu$-bubble (see Lemma \ref{densest}(2) below) and the other is
the second variation for the unequally warped $\mu$-bubble
(see \eqref{eqn:Defnc0} and \eqref{eq:2ndvar_final} below). At present
we do not know if such condition can be removed or not.

(iii) When $0<N<\infty$, Yeung's parameter range   $\gamma<\Big[\tfrac{1}{(n-1)(1+\frac{n-1}{N})}+\tfrac{n-1}{4}\Big]^{-1}$
is wider than our parameter range $\gamma<\left(\frac{1}{n-1}+\frac{n-1}{4}\right)^{-1}$,
but Yeung's spectrum condition \eqref{specNcon} is stronger than our spectrum
condition \eqref{speccond}. Hence our result is a supplement of Yeung's result
in weighted manifolds.

(iv) When $f$ is constant, Theorem \ref{result2} reduces to Antonelli-Pozzetta-Xu's
result \cite{APX24} but our parameter range
$\gamma<\left(\frac{1}{n-1}+\frac{n-1}{4}\right)^{-1}$ is narrower than their
parameter range $\gamma<\frac{4}{n-1}$.
\end{remark}

The proof of Theorem \ref{result2} applies the trick of Antonelli-Pozzetta-Xu
\cite{APX24}, by combining the $\mu$-bubble technique in \cite{Gromov19,Zhu23}
and the surface-capturing technique in \cite{CCE16,CEM19,CL24,Liu,Zhu23}. In
our proof, we follow the argument of \cite{APX24} by studying an unequally
warped $\mu$-bubble in weighted manifolds. Notice that Yeung \cite{Yeu}
also followed the argument of \cite{APX24} to study the spectrum splitting
theorem in weighted manifolds. His proof considers the energy functional of
finite perimeter subsets when $\Ric^N_f$ is nonnegative in the spectrum sense,
whereas our proof focuses on such functional when $\Ric_f$ is nonnegative
in the spectrum sense.

The rest paper is organized as follows. In Section \ref{Pre}, we give some basic
concepts about sets of locally finite perimeter in manifolds. In Section \ref{diacomp},
we prove the diameter comparison in Theorem \ref{result}. In Section \ref{volcomp},
we prove the global volume comparison in Theorem \ref{result}. In Section
\ref{split}, we prove Theorem \ref{result2}.


\section{Preliminary}\label{Pre}
In this section, we recall some basic concepts about sets of finite perimeter
in manifolds (see e.g. \cite{Vo}). In particular, we give the De Giorgi's
structure theorem in manifolds.
We refer the reader to \cite{AFP00, Maggi12} for the classical geometric measure
theory of sets of finite perimeter in the Euclidean space, which naturally extends
to the manifold case.

If $E$ is a Lebesgue measurable set in an $n$-dimensional Riemannian manifold $(M,g)$,
then the set $E$ is called \emph{locally finite perimeter} if there exists a Radon
measure $v_E$ in $(M,g)$, called the Gauss-Green measure of $E$, such that
\[
\int_E \nabla\varphi(x) dv_x=\int_M \varphi dv_E
\]
holds for all $\varphi\in C_c^1(M)$, where $C_c^1(M)$ denotes the space of
differentiable function with compact support in $M$. The total variation
measure $|v_E|$ of $v_E$ induces the notions of both the relative perimeter
measure $P(E,A)$ of $E$ with respect to a Borel set $A\subset M$, defined by
\[
P(E,A)=|v_E|(A)
\]
and the total perimeter measure $P(E)$ of $E$, defined by
\[
P(E)=|v_E|(M).
\]
In particular, $E$ is a set of finite perimeter measure if and only if
$P(E)<\infty$. If $E$ has a $C^1$ boundary, then the above concepts
return to the classical Hausdorff measure. In particular,
\[
P(E,A)=\mathcal{H}^{n-1}(A \cap \p E)
\]
and
\[
P(E)=\mathcal{H}^{n-1}(\p E),
\]
where $\mathcal{H}^{n-1}$ denotes the $(n-1)$-dimensional Hausdorff measure.

On an $n$-dimensional Riemannian manifold $(M,g)$, an (outer) measure $\mu$
on $M$ is concentrated on $E\subset M$ if $\mu(M\setminus E)=0$. The intersection
of the closed sets $E$ such that $\mu$ is concentrated on $E$ is denoted by
$\mathrm{supp}\,\mu$, and called the support of $\mu$. The \emph{reduced boundary}
$\p^*E$ of a set of locally finite perimeter $E$ in $M$ is the set of those
$x\in \mathrm{supp}\,v_E$ such that the limit
\[
\lim_{r\to 0+}\frac{ v_E(B(x,r))}{|v_E|(B(x,r)}
\]
exists and belongs to $\S^{n-1}$. We say that the \emph{measure-theoretic outer
unit normal} to $E$, denoted by $\nu_E$, is a Borel function $\nu_E:\p^*E\to\S^{n-1}$
by setting
\[
\nu_E(x)=\lim_{r\to 0+}\frac{ v_E(B(x,r))}{|v_E|(B(x,r)},
\]
where $x\in\p^*E$. The reduced boundary $\p^*E$ is uniquely determined by
the Gauss–Green measure $v_E$ of $E$. If $E$ is an open set with $C^1$ boundary,
then $\p^*E=\p E$ and the measure-theoretic outer unit normal coincides with
classical notion of outer unit normal.

By the well-known De Giorgi's structure theorem (see e.g. Theorem 15.9 in
\cite{Maggi12}), we know that the perimeter is concentrated on the reduced
boundary $\p^*E$. More precisely, if $E$ is a set of locally finite perimeter
in an $n$-dimensional Riemannian manifold $(M,g)$, then the Gauss–Green measure
$v_E$ of $E$ satisfies
\[
v_E=\nu_E\, \mathcal{H}^{n-1}\llcorner\p^*E,\quad
|v_E|=\mathcal{H}^{n-1}\llcorner\p^*E
\]
and the generalized Gauss-Green formula
\[
\int_E \nabla\varphi(x)dv_x=\int_{\p^*E} \varphi\,\nu_E\, d\mathcal{H}^{n-1}
\]
holds for all $\varphi\in C_c^1(M)$. These results are motivated by the classical
Gauss-Green theorem.

On a Riemannian manifold $(M,g)$, for any set of locally finite perimeter $E\subset M$
and for any Borel set $A\subset M$, in this paper, we will denote
\[
|\partial^* E\cap A|:=P(E,A)\quad\mathrm{and}\quad
|\partial^* E|:=P(E).
\]
for ease of notations. These notations will be repeatedly used in Section \ref{split}.

\section{Diameter comparison}\label{diacomp}

In this section, we shall prove the first part of Theorem \ref{result} by following
the arguments of \cite{AX17,CH}.

On $(M,g,e^{-f}dv)$, let $\Omega_-\subset\Omega_+\subset M$ be two domains with
non-empty boundaries
such that $\overline{\Omega_+}\setminus\Omega_-$ is compact. Let $h\in C^\infty(\Omega_+\setminus\overline{\Omega_-})$ be a smooth function such that
\[
\text{$\lim_{x\to\de\Omega_-}h(x)=+\infty$ and $\lim_{x\to\de\Omega_+}h(x)=-\infty$ uniformly.}
\]
Fix constants $k$, $\alpha$, $\gamma\in[0,1]$ and a domain $\Omega_0$
with $\Omega_-\Subset\Omega_0\Subset\Omega_+$. Here $\Omega_-\Subset\Omega_0$
means that the closure of $\Omega_-$ is relatively compact in $\Omega_0$.
For a set $\Omega\subset M$ of locally finite perimeter, we consider the functional
\[
E_f(\Omega):=\int_{\p^*\Omega}u^{\gamma}e^{-f}-\int_M(\chi_{\Omega}-\chi_{\Omega_0})hu^\alpha e^{-(k+1)f},
\]
where $\p^*\Omega$ is the reduced boundary of $\Omega$.  Here $u\in C^{\infty}(M)$
is a positive function  satisfying $u\,\Ric_f-\gamma\Delta_fu \geq (n-1)\lambda u$.
In the above (and following)
integrals, we often omit the area measure element $d\mathcal{H}^{n-1}$ and the
volume measure element $dv$ for simplicity. Here $\chi_\Omega$ denotes the
characteristic function on $\Omega$. By a similar argument of Proposition 2.1 in
\cite{Zhu21}, the functional $E_f$ must have a minimizer $\Omega$ such that $\Omega_-\Subset\Omega\Subset\Omega_+$. By the classical geometric measure
theory (see e.g. Theorems 27.5 and 28.1 in \cite{Maggi12}), we see that
$\p\Omega$ is smooth when $n\leq 7$ while the singular part of $\p\Omega$
has Hausdorff dimension no more than $n-8$.

We first consider the case $3\le n\le 7$. For any $\vp\in C^{\infty}(M)$,
let $\{\Omega_t\}_{t\in(-\e,\e)}$ be a smooth family of sets
with $\Omega_0=\Omega$ and $\frac{\p \Omega_t}{\de t}=\vp\nu$ at $t=0$,
where $\nu$ is the outer unit normal of $\p\Omega_t$. The first variation
of $E_f(\Omega_t)$ is that
\[
0=\frac{\mathrm{d}E_f(\Omega_t)}{\mathrm{d}t}\Bigg|_{t=0}
=\int_{\p\Omega}\Big(H_f+\gamma u^{-1}u_\nu-hu^{\alpha-\gamma}e^{-kf}\Big)u^\gamma e^{-f}\vp,
\]
where $H_f:=H-g(\nabla f,\nu)=H-f_\nu$. Since $\vp$ is arbitrary,
by letting $\vp=H_f+\gamma u^{-1}u_\nu-hu^{\alpha-\gamma}e^{-kf}$,
the above variation implies
\begin{equation}\label{H}
H_f=hu^{\alpha-\gamma}e^{-kf}-\gamma u^{-1}u_\nu \quad \text{on $\p \Omega$.}
\end{equation}
The second variation of $E_f(\Omega_t)$ is that
\begin{equation}
\begin{split}\label{varia2nd}
0\le&\frac{\mathrm{d}^2 E_f(\Omega_t)}{\mathrm{d}t^2}\Bigg|_{t=0}\\
=&\int_{\p\Omega}\Big[-\Delta_{\p\Omega}\vp-|\mathrm{II}|^2\vp
-\Ric_f(\nu,\nu)\vp-\gamma u^{-2}u_\nu^2\vp\\
&\quad\quad \ \,+\gamma u^{-1}\vp(\Delta u-\Delta_{\p\Omega}u-Hu_\nu)
-\gamma u^{-1}\langle\D_{\p\Omega}u,\D_{\p\Omega}\vp\rangle
+\langle\D_{\p\Omega} f,\D_{\p\Omega}\vp\rangle \\
&\quad \quad \ \,-h_{\nu}u^{\alpha-\gamma}e^{-kf}\vp
-(\alpha-\gamma)hu^{\alpha-\gamma-1}u_\nu e^{-kf}\vp
+k hu^{\alpha-\gamma}e^{-kf}f_{\nu}\varphi\Big] u^\gamma e^{-f}\vp\\
=&\int_{\p\Omega}
\Big[-\Delta_{\p\Omega}\vp-|\mathrm{II}|^2\vp
+\big(\gamma\Delta_fu-u\Ric_f(\nu,\nu)\big)u^{-1}\varphi -\gamma u^{-2}u_\nu^2\vp\\
&\quad\quad \,+\gamma u^{-1}\vp(-\Delta_{\p\Omega}u+\langle\D u,\D f\rangle-Hu_\nu)
-\gamma u^{-1}\langle\nabla_{\p\Omega}u,\nabla_{\p\Omega}\vp\rangle
+\langle\D_{\p\Omega} f,\D_{\p\Omega}\varphi\rangle \\
&\quad\quad \,-h_{\nu}u^{\alpha-\gamma}e^{-kf}\vp-(\alpha-\gamma)hu^{\alpha-\gamma-1}u_\nu e^{-kf}\vp
+khu^{\alpha-\gamma}e^{-kf}f_{\nu}\vp\Big] u^\gamma e^{-f}\vp,
\end{split}
\end{equation}
where we used the definitions of $\Delta_{f}$ and $\Ric_f$.

If we choose some special case of $\vp$ and $\alpha$, \eqref{varia2nd}
can be simplified as follows.
\begin{lemma}\label{h inequality}
Let $\vp=u^{-\gamma}$ and $\alpha=k\gamma$ in \eqref{varia2nd}. Then
\[
0\le\int_{\p\Omega}\Big[|\D h|u^{\alpha-\gamma}-\frac{4k-(n-1)k^2}{4}u^{2\alpha-2\gamma}e^{-kf}h^2
-(n-1)\lambda e^{kf}\Big]u^{-\gamma}e^{-(k+1)f}.
\]
\end{lemma}

\begin{proof}[Proof of Lemma \ref{h inequality}]
Letting $\vp=u^{-\gamma}$ and $\gamma\ge0$, we have
\begin{equation}\label{estimate1}
-\Delta_{\p\Omega}\varphi-\gamma u^{-1}\vp\Delta_{\p\Omega}u
-\gamma u^{-1}\langle\D_{\p\Omega}u,\D_{\p\Omega}\vp\rangle
= -\gamma u^{-\gamma-2}|\nabla_{\de\Omega}u|^2\le 0
\end{equation}
and
\begin{equation}\label{estimate2}
\gamma u^{-1}\varphi\langle\D u,\D f\rangle
+\langle\D_{\p\Omega}f,\D_{\p\Omega}\vp\rangle
=\gamma u^{-\gamma-1}u_{\nu}f_{\nu}.
\end{equation}
Using \eqref{estimate1}, \eqref{estimate2}, the trace inequality $|\mathrm{II}|^2\ge\frac{H^2}{n-1}$
and $u\,\Ric_f-\gamma\Delta_fu\ge(n-1)\lambda u$, the inequality \eqref{varia2nd} can be
reduced to
\[
\begin{split}
0\le\int_{\p\Omega}\Big[&-\frac{H^2}{n-1}u^{-\gamma}-(n-1)\lambda u^{-\gamma}
-\gamma u^{-\gamma-2}u_\nu^2-\gamma u^{-\gamma-1}H_fu_\nu
-h_{\nu}u^{\alpha-\gamma}e^{-kf}u^{-\gamma}\\
&-(\alpha-\gamma)hu^{\alpha-\gamma-1}u_\nu e^{-kf}u^{-\gamma}
+khu^{\alpha-\gamma}e^{-kf}f_{\nu}u^{-\gamma}\Big]e^{-f} \\
\le\int_{\partial\Omega}\Big[&-\frac{H^2}{n-1}-(n-1)\lambda-\gamma u^{-2}u_\nu^2
-\gamma u^{-1}H_fu_\nu+|\D h|u^{\alpha-\gamma}e^{-kf}\\
&-(\alpha-\gamma)hu^{\alpha-\gamma-1}u_\nu e^{-kf}
+khu^{\alpha-\gamma}e^{-kf}f_{\nu}\Big]u^{-\gamma}e^{-f}.
\end{split}
\]
Let
\[
X=hu^{\alpha-\gamma}e^{-kf}\quad \mathrm{and}\quad
Y=u^{-1}u_{\nu}.
\]
By \eqref{H}, then
\[
H_f=X-\gamma Y.
\]
Thus the above inequality becomes
\[
\begin{split}
0\le\int_{\p\Omega}\bigg[&-\frac{(f_{\nu}+X-\gamma Y)^2}{n-1}-(n-1)\lambda-\gamma Y^2-\gamma Y(X-\gamma Y)\\
&+|\D h|u^{\alpha-\gamma}e^{-kf}+(\gamma-\alpha)XY+kXf_\nu\bigg] u^{-\gamma}e^{-f} \\
= \int_{\p\Omega}
\bigg[&-\frac{f_{\nu}^2}{n-1}+\Big(k-\frac{2}{n-1}\Big)Xf_{\nu}+\frac{2\gamma}{n-1}Yf_{\nu}-\frac{X^{2}}{n-1} \\
&+\Big(\frac{2\gamma}{n-1}-\alpha\Big)XY+\Big(\frac{n-2}{n-1}\gamma^2-\gamma\Big)Y^2
+|\D h|u^{\alpha-\gamma}e^{-kf}-(n-1)\lambda\bigg]u^{-\gamma}e^{-f}.
\end{split}
\]
Noticing that
\[
\begin{split}
-\frac{f_{\nu}^2}{n-1}&+\Big(k-\frac{2}{n-1}\Big)Xf_{\nu}+\frac{2\gamma}{n-1}Yf_{\nu}\\
=& -\frac{f_{\nu}^2}{n-1}+\left[\Big(k-\frac{2}{n-1}\Big)X+\frac{2\gamma}{n-1}Y\right]f_{\nu}\\
\le& \frac{n-1}{4}\left[\Big(k-\frac{2}{n-1}\Big)X+\frac{2\gamma}{n-1}Y\right]^2\\
=&\frac{\big((n-1)k-2\big)^2}{4(n-1)}X^2
+\frac{\big((n-1)k-2\big)\gamma}{n-1}XY+\frac{\gamma^2}{n-1}Y^2,
\end{split}
\]
so we have
\[
0\le\int_{\p\Omega}\bigg[\tfrac{(n-1)k^2-4k}{4}X^2+(k\gamma{-}\alpha)XY+\gamma(\gamma-1)Y^2
+|\D h|u^{\alpha-\gamma}e^{-kf}-(n-1)\lambda\bigg]u^{-\gamma}e^{-f}.
\]
Since $\alpha=k\gamma$ and $0\le\gamma\le 1$, we further have
\[
0\le\int_{\partial\Omega}\bigg[-\frac{4k-(n-1)k^2}{4}X^2+|\D h|u^{\alpha-\gamma}e^{-kf}-(n-1)\lambda\bigg]u^{-\gamma}e^{-f}.
\]
Since $X=hu^{\alpha-\gamma}e^{-kf}$, the lemma follows.
\end{proof}

To prove the diameter comparison, we also need the following lemma
which has been confirmed in the proof of Lemma 1 of \cite{AX17}.
\begin{lemma}\label{h lemma}
Let $(M,g)$ be a Riemannian manifold and $C$, $D$ be two positive constants.
If
\[
\diam(M,g)>\frac{\pi}{\sqrt{CD}},
\]
then there exist two domains $\Omega_-\subset\Omega_+\subset M$ with non-empty
boundaries such that $\overline{\Omega_+}\setminus\Omega_-$ is compact and
function $h\in C^{\infty}(\Omega_+\setminus\ov{\Omega_-})$ satisfies
\[
|\D h|<Ch^{2}+D
\]
with $\lim_{x\to\de\Omega_-}h(x)=+\infty$ and $\lim_{x\to\de\Omega_+}h(x)=-\infty$
uniformly.
\end{lemma}

Using the above lemmas, we shall prove the diameter comparison in
Theorem \ref{result}.
\begin{proof}[Proof of (1) in Theorem \ref{result}]
We first prove the result when $3\le n\le 7$.
Let $0<k\le 1$ and $4k-(n-1)k^2>0$.
Then $\alpha=k\gamma\le\gamma$. Define two positive constants
\[
C=\frac{4k-(n-1)k^2}{4}\cdot\frac{(\sup_Mu)^{2\alpha-2\gamma}}{(\inf_Mu)^{\alpha-\gamma}} e^{-kf_{max}}
\quad\text{and}\quad
D=(n-1)\lambda\left(\inf_Mu\right)^{\gamma-\alpha}e^{kf_{min}}.
\]
From Lemma \ref{h inequality}, we have that
\[
\begin{split}
0\le&\int_{\p\Omega}\Big[|\D h|u^{\alpha-\gamma}-C(\inf_Mu)^{\alpha-\gamma}h^{2}
-D(\inf_Mu)^{\alpha-\gamma}\Big]u^{-\gamma}e^{-(k+1)f} \\
\le&\int_{\p\Omega}\Big[|\D h|-Ch^2-D\Big]u^{\alpha-2\gamma}e^{-(k+1)f}.
\end{split}
\]
If $\diam(M,g)>\frac{\pi}{\sqrt{CD}}$, then we get a contradiction by letting
function $h$ as in Lemma \ref{h lemma}. Hence we conclude that
$\diam(M,g)\le\frac{\pi}{\sqrt{CD}}$. According to the definitions of $C$ and $D$,
we in fact get
\begin{equation}\label{genest}
\diam(M,g)\le\left(\frac{\sup_Mu}{\inf_Mu}\right)^{\gamma-\alpha}
\cdot\sqrt{\frac{4}{n-1}}\cdot\sqrt{\frac{e^{k(f_{max}-f_{min})}}{4k-(n-1)k^2}}
\cdot\frac{\pi}{\sqrt{\lambda}},
\end{equation}
where $\alpha=k\gamma$, $0\le \gamma\le 1$ and $k\leq1$, $4k-(n-1)k^2>0$.
Taking $k=\frac{2}{n-1}$, then
\[
\diam(M,g)\le\left(\frac{\sup_Mu}{\inf_Mu}\right)^{\frac{n-3}{n-1}\gamma}
\cdot e^{\frac{f_{max}-f_{min}}{n-1}}\cdot\frac{\pi}{\sqrt{\lambda}}.
\]

In the rest we prove the conclusion when $n\ge8$. In this case, $\de\Omega$
may have singularities. The diameter comparison can be proved by modifying
the argument of the case $3\le n\le7$ verbatim as Appendix A in \cite{AX17}.
\end{proof}

From the above proof course, if we do not choose a special value of $k$, we
indeed prove the following general result, which may be a little better than
Theorem \ref{result}(1) in some cases.
\begin{theorem}\label{resultk}
Let $(M,g,e^{-f}dv)$ be a complete $n$-dimensional ($n\geq3$) weighted manifold
and $\gamma$ be a constant such that $0\le \gamma\le 1$. Assume that there exists
a positive function $u\in C^{\infty}(M)$ such that
\[
u\,\Ric_f-\gamma\Delta_fu \geq (n-1)\lambda u
\]
for some constant $\lambda>0$. If $0<\inf_Mu\le\sup_Mu<\infty$ and
$F=\|f\|_{C^0(M)}<\infty$, then the following diameter comparison holds:
\[
\diam(M,g)\le\left(\frac{\sup_Mu}{\inf_Mu}\right)^{(1-k)\gamma}
\cdot\sqrt{\frac{e^{k(f_{max}-f_{min})}}{4k-(n-1)k^2}}
\cdot\frac{2\pi}{\sqrt{(n-1)\lambda}},
\]
where $0<k\le 1$ and $4k-(n-1)k^2>0$.
Here $f_{max}:=\sup_{x\in M} f(x)$ and $f_{min}:=\inf_{x\in M} f(x)$.
\end{theorem}
In particular, if we let $u\equiv 1$ in Theorem \ref{resultk} and
get the following diameter estimate of the weighted manifold.
\begin{corollary}\label{cordia}
Let $(M,g,e^{-f}dv)$ be a complete $n$-dimensional ($n\geq3$) weighted manifold
satisfying $\Ric_f\ge(n-1)\lambda $ for some constant $\lambda>0$. If
$F=\|f\|_{C^0(M)}<\infty$, then
\begin{equation}\label{genest}
\mathrm{diam}(M,g)\le\sqrt{\frac{e^{k(f_{max}-f_{min})}}{(n-1)k-\frac{(n-1)^2}{4}k^2}}\cdot\frac{\pi}{\sqrt{\lambda}},
\end{equation}
where $0<k\le 1$ and $4k-(n-1)k^2>0$. Here $f_{max}:=\sup_{x\in M} f(x)$ and $f_{min}:=\inf_{x\in M} f(x)$.
\end{corollary}

Under the same assumption of Corollary \ref{cordia},  Wei-Wylie \cite{WW09}
proved that $\mathrm{diam}(M,g)\le(1+\frac{4F}{(n-1)\pi})\cdot\frac{\pi}{\sqrt{\lambda}}$;
Limoncu \cite{Lim12} proved that $\mathrm{diam}(M,g)\le \sqrt{1+\frac{2\sqrt{2}F}{n-1}}\cdot\frac{\pi}{\sqrt{\lambda}}$;
 Tadano \cite{Tad16} proved that $\mathrm{diam}(M,g)\le \sqrt{1+\frac{8 F}{(n-1)\pi}}\cdot\frac{\pi}{\sqrt{\lambda}}$. Compared with these
estimates, our estimate is possibly sharper than them in some cases.

\section{Global weighted volume comparison}\label{volcomp}
In this section, we prove the second part of Theorem \ref{result} by following
the arguments of \cite{AX17,CH}. To achieve this goal, we first recall the following
ODE comparison appeared in Lemma 4 of \cite{AX17}.
\begin{lemma}\label{ODE comparison}
If $J:[0,V_{1})\to\R$ is a continuous function, where $V_1>0$ is a
constant, such that
\begin{itemize}
\item[(1)] $J(0)=0$ and $J(v)>0$ for any $v\in(0,V_1)$,
\item[(2)] $\limsup_{v\to0^+}v^{-\frac{n-1}{n}}J(v)\le n|\mathbb{B}^n|^{\frac 1n}$,
where $|\mathbb{B}^n|$ is the Riemannian volume of the unit ball $\mathbb{B}^n$
in $\R^n$,
\item[(3)] $J''J\le-\frac{(J')^2}{n-1}-(n-1)\Lambda$ in the viscosity sense on
$(0,V_1)$ for some constant $\Lambda>0$,
\end{itemize}
then
\[
V_1\le\Lambda^{-\frac n2}|\S^n|.
\]
\end{lemma}

Without loss of generality, we assume that $\inf_Mu=1$. We set
\[
\alpha=\frac{2\gamma}{n-1}\quad \mathrm{and}\quad k=\frac{2}{n-1}.
\]
By Theorem \ref{result}(1), $M$ is compact and hence
\[
V_0:=\int_Mu^\alpha e^{-(k+1)f}<\infty.
\]
Define the following unequally weighted isoperimetric profile
\begin{equation}\label{isoprofile}
I(v):=\inf\left\{\int_{\p^*E}u^\gamma e^{-f}\,\Big|\,\text{$E\Subset M$
has finite perimeter and $\int_Eu^\alpha e^{-(k+1)f}=v$}\right\}
\end{equation}
for all $v\in[0,V_0)$, where $\p^*E$ is the reduced boundary of $E$. Then
\begin{equation}\label{ODE comparison condition a}
I(0)=0 \quad\text{and}\quad I(v)>0\quad \text{for} \,\, v\in(0,V_0).
\end{equation}
By the argument of Proposition 5.3 in \cite{CLMS24}, $I$ is continuous on $[0,V_0)$.
Moreover, we have

\begin{lemma}\label{limestiI}
Under the same assumption of Theorem \ref{result}, there exists
a finite perimeter set $B_r(x_0)\Subset M$, where $x_0\in M$
and $r>0$, such that $I$ defined as in \eqref{isoprofile} satisfies
\begin{equation}\label{ODE comparison condition b}
\displaystyle{\limsup_{v\to0^+}}\,v^{-\frac{n-1}{n}}I(v)
\le n|\mathbb{B}^n|^{\frac 1n}e^{\frac{1}{n}f_{max}}.
\end{equation}
\end{lemma}
\begin{proof}[Proof of Lemma \ref{limestiI}]
Since $\inf_Mu=1$ and $M$ is compact, then there exists a point $x_0\in M$
satisfying $u(x_0)=1$. When $r\ll1$, we have the asymptotical expansion
\[
\int_{B(x_0,r)}e^{-(k+1)f}=|\mathbb{B}^n|r^ne^{-(k+1)f(x_0)}+O(r^{n+1})
\]
and
\[
\int_{\de B(x_0,r)} e^{-f}=n|\mathbb{B}^n|r^{n-1}e^{-f(x_0)}+O(r^{n}).
\]
From the definition of $I$, when $v\ll1$, we have
\[
\begin{split}
I(v)\le& n|\mathbb{B}^n|^{\frac 1n}v^{\frac{n-1}{n}}e^{\frac{(n-1)(k+1)-n}{n}f(x_0)}+o(v^{\frac{n-1}{n}}) \\
=& n|\mathbb{B}^n|^{\frac 1n}v^{\frac{n-1}{n}}e^{\frac 1nf(x_0)}+o(v^{\frac{n-1}{n}}) \\
\le& n|\mathbb{B}^n|^{\frac 1n}v^{\frac{n-1}{n}}e^{\frac{1}{n}f_{max}}+o(v^{\frac{n-1}{n}}),
\end{split}
\]
where we used $k=\frac{2}{n-1}$ in the above second line. This implies
\eqref{ODE comparison condition b}.
\end{proof}

For each $v_0\in(0,V_0)$, we let $E$ be the volume-constrained minimizer of $v_0$, i.e.
\begin{equation}\label{volconstr}
\int_Eu^\alpha e^{-(k+1)f}=v_0\quad \text{and} \quad
\int_{\p^*E}u^\gamma e^{-f}=I(v_0).
\end{equation}
By the classical geometric measure theory (see e.g. Section 3.10 in \cite{Mor}),
$\de E$ is smooth when $n\leq 7$ while the singular part of $\de E$
has Hausdorff dimension no more than $n-8$.

\smallskip

For the convenience of discussion, let us first consider the case $3\le n\le7$.
We may show that the above profile $I$ satisfies a good differential inequality.

\begin{lemma}\label{ODE comparison condition c}
Let $(M^n, g,e^{-f}dv)$ be a complete weighted manifold.
Let $\gamma,\alpha,V_0,I(v)$ be as above and let $\lambda>0$.
Assume that $u$ is a smooth function on $M$ satisfying
\[
\inf_M u=1 \quad \mathrm{and} \quad
u\,\Ric_f-\gamma\Delta_fu \geq (n-1)\lambda u.
\]
Suppose for a fixed $v_0\in(0,V_0)$, there exists a bounded set $E$ with
finite perimeter satisfying \eqref{volconstr}. If $F:=\|f\|_{C^0(M)}<\infty$,
then $I$ satisfies
\[
I''I\le-\frac{(I')^2}{n-1}-(n-1)\lambda e^{2kf_{min}}
\]
in the viscosity sense at $v_0$.
\end{lemma}
\begin{proof}[Proof of Lemma \ref{ODE comparison condition c}]
For $v_0\in(0,V_0)$, we let $E$ be the volume-constrained minimizer of $v_0$.
For any $\vp\in C^{\infty}(M)$, we let $\{E_t\}_{t\in(-\e,\e)}$ be a smooth
family of sets satisfying $E_0=E$ and $\frac{\de E_t}{\de t}=\vp\nu$ at
$t=0$, where $\nu$ denotes the outer unit normal vector of $\de E_t$. Set
\[
V(t):=\int_{E_t}u^\alpha e^{-(k+1)f}\quad  \text{and} \quad
A(t):=\int_{\p E_t}u^\gamma e^{-f}.
\]
We directly compute that
\[
V'(0)=\int_{\p E}u^\alpha e^{-(k+1)f}\vp,
\]
\[
V''(0)=\int_{\p E}\left(H_f+\alpha u^{-1}u_{\nu}-kf_{\nu}\right)u^\alpha e^{-(k+1)f}\vp^2
+u^{\alpha}e^{-(k+1)f}\vp\vp_{\nu}
\]
and
\[
A'(0)=\int_{\p E}(H_f+\gamma u^{-1}u_\nu)u^\gamma e^{-f}\vp,
\]
\[
\begin{split}
A''(0)=\int_{\p E}\Big[&-\Delta_{\p E}\vp-\Ric_f(\nu,\nu)\vp-|\mathrm{II}|^2\vp
+\langle\D_{\de E}f,\D_{\p E}\vp\rangle\\
&+\gamma u^{-1}\vp(\Delta u-\Delta_{\p E}u-Hu_\nu)-\gamma u^{-2}u_\nu^2\vp
-\gamma u^{-1}\langle\D_{\p\Omega}u,\D_{\p\Omega}\vp\rangle\Big]u^\gamma e^{-f}\vp\\[1mm]
&+\big(H_f+\gamma u^{-1}u_\nu\big)\big(\gamma u^{-1}u_\nu \vp+\vp_\nu+H_f\vp\big)u^{\gamma}e^{-f}\vp.
\end{split}
\]
By the preceding calculation, we observe that
\[
-\Ric_f(\nu,\nu)\varphi+\gamma u^{-1}\varphi\Delta u
=\big(\gamma\Delta_fu-u\,\Ric_f(\nu,\nu)\big)u^{-1}\varphi+\gamma u^{-1}\varphi\langle\D u,\D f\rangle.
\]
So we can further simplify the expression of $A''(0)$ as
\[
\begin{split}
A''(0)=\int_{\p E}\Big[&-\Delta_{\p E}\vp-|\mathrm{II}|^{2}\vp+\big(\gamma\Delta_{f}u-u\Ric_{f}(\nu,\nu)\big)u^{-1}\varphi\\
&+\gamma u^{-1}\varphi\langle\D u,\D f\rangle+\langle\D_{\p E}f,\D_{\p E}\vp\rangle-\gamma u^{-2}u_\nu^2\vp\\
&+\gamma u^{-1}\vp(-\Delta_{\p E}u-Hu_\nu)-\gamma u^{-1}\langle\D_{\p\Omega}u,\D_{\p\Omega}\vp\rangle \\
&+(H_f+\gamma u^{-1}u_\nu\big)\big(\gamma u^{-1}u_\nu \vp+\vp_\nu+H_f\vp)\Big]u^\gamma e^{-f}\vp.
\end{split}
\]

Since $E$ is a volume-constrained minimizer, we see that if
\[
V'(0)=\int_{\p E}u^\alpha e^{-(k+1)f}\varphi=0,
\]
then
\[
A'(0)=\int_{\p E}\mathcal{H}_f\cdot u^\alpha e^{-(k+1)f}\vp=0.
\]
where $\mathcal{H}_f:=(H_f+\gamma u^{-1}u_\nu)u^{\gamma-\alpha}e^{kf}$.
Since $\vp$ is arbitrary, if we particularly let
\[
u^\alpha e^{-(k+1)f}\varphi=\mathcal{H}_f-\overline{\mathcal{H}}_f,
\]
where $\overline{\mathcal{H}}_f=\frac{1}{|\p E|}\int_{\p E}\mathcal{H}_f$,
then $V'(0)=0$ always holds and $A'(0)=0$ gives that
\[
\overline{\mathcal{H}_f}\int_{\p E}\mathcal{H}_f=\int_{\p E}\mathcal{H}_f^2
\ge \frac{1}{|\p E|}\left(\int_{\p E}\mathcal{H}_f\right)^2,
\]
where we used the Cauchy-Schwarz inequality in the second inequality.
Combining this with the definition of $\overline{\mathcal{H}}_f$, we
see that the above inequality actually becomes equality and hence $\overline{\mathcal{H}}_f$ is
constant on $\p E$. This shows that
\[
\text{$u^{\gamma-\alpha}(H_f+\gamma u^{-1}u_\nu)e^{kf}$ is constant on $\p E$.}
\]

Now, we choose $\vp=u^{-\gamma}$ and then
\[
V'(0)=\int_{\p E}u^{\alpha-\gamma}e^{-(k+1)f}>0.
\]
Hence $V(t)$ has the inverse function in a neighbourhood of $0$, and we can define
\[
G(v):=A(V^{-1}(v)).
\]
Noticing that
\[
G(v_0)=I(v_0)\quad \text{and} \quad G(v)\ge I(v)
\]
in some neighbourhood of $v_0$, so function $G$ is an upper barrier of $I$ at $v_0$.
To prove Lemma \ref{ODE comparison condition c}, it suffices to show
\begin{equation}\label{goal G}
G(v_{0})G''(v_{0}) \leq -\frac{G'^{2}(v_{0})}{n-1}-(n-1)\lambda e^{2kf_{min}}.
\end{equation}

By the chain rule and the formula for the derivative of inverse function,
we observe that
\begin{equation}\label{G'}
G'(v_0)=A'(0)(V^{-1})'(v_0)=\frac{A'(0)}{V'(0)}=u^{\gamma-\alpha}(H_f+\gamma u^{-1}u_\nu)e^{kf},
\end{equation}
where we used the fact that $u^{\gamma-\alpha}(H_f+\gamma u^{-1}u_\nu)e^{kf}$
is constant on $E$, and
\[
G''(v_0)=\frac{A''(0)}{V'^2(0)}-\frac{A'(0)V''(0)}{V'^3(0)}
=\frac{A''(0)-G'(v_0)V''(0)}{V'^2(0)}.
\]

From above, to prove \eqref{goal G}, we first compute $V''(0)$ and $A''(0)$.
It is clear that
\[
V''(0)=\int_{\p E}\Big(H_f-kf_{\nu}+(\alpha-\gamma)u^{-1}u_{\nu}\Big)u^{\alpha-2\gamma}e^{-(k+1)f}.
\]
For $A''(0)$, the same calculation as \eqref{estimate1} and \eqref{estimate2} gives that
\[
-\Delta_{\p E}\varphi-\gamma u^{-1}\varphi\Delta_{\p E}u
-\gamma u^{-1}\langle\D_{\p E} u,\D_{\p E}\vp\rangle
=-\gamma u^{-\gamma-2}|\D_{\p E}u|^2
\le 0
\]
and
\[
\gamma u^{-1}\varphi\langle\D u,\D f\rangle+\langle\D_{\p E}f,\D_{\p E}\varphi\rangle
=\gamma u^{-\gamma-1}u_{\nu}f_{\nu}.
\]
Combining the above two relations with the trace inequality $|\mathrm{II}|^2\ge\frac{H^2}{n-1}$
and $u\,\Ric_f-\gamma\Delta_fu \geq (n-1)\lambda u$, we hence get
\[
\begin{split}
A''(0)&\le\int_{\p E}\Big[\frac{-H^2u^{-\gamma}}{n-1}{-}(n-1)\lambda u^{-\gamma}
{-}\gamma u^{-\gamma-2}u_\nu^2{-}\gamma u^{-\gamma-1}H_fu_\nu
{+}H_fu^{-\gamma}(H_f{+}\gamma u^{-1}u_\nu)\Big]e^{-f}\\
&=\int_{\p E}\Big[\frac{-H^2}{n-1}-(n-1)\lambda-\gamma u^{-2}u_\nu^2
-\gamma u^{-1}H_fu_\nu+H_f(H_f+\gamma u^{-1}u_\nu)\Big] u^{-\gamma}e^{-f}.
\end{split}
\]

In the following, we will apply the above formulas to check \eqref{goal G}.
For simplicity, let
\[
X=G'(v_0)u^{\alpha-\gamma}e^{-kf}\quad \mathrm{and}\quad  Y=u^{-1}u_\nu.
\]
Then \eqref{G'} implies
\[
H_f=X-\gamma Y.
\]
We compute that
\[
\begin{split}
G'(v_0)V''(0)=&\int_{\p E}u^{\gamma-\alpha}e^{kf}X\cdot\Big(X-kf_{\nu}
+(\alpha-2\gamma)Y\Big)u^{\alpha-2\gamma}e^{-(k+1)f}\\
=&\int_{\p E}\Big(X^2-kXf_{\nu}+(\alpha-2\gamma)XY\Big)u^{-\gamma}e^{-f}
\end{split}
\]
and
\[
\begin{split}
A''(0)\le\int_{\p E}\bigg[&-\frac{(f_{\nu}+X-\gamma Y)^2}{n-1}-(n-1)\lambda+\gamma Yf_{\nu}\\
&-\gamma Y^2-\gamma Y(f_{\nu}+X-\gamma Y)+X(X-\gamma Y)\bigg]u^{-\gamma}e^{-f} \\
=\int_{\p E}\bigg[&-(n-1)\lambda-\frac{1}{n-1}f_{\nu}^2-\frac{2}{n-1}Xf_{\nu}
+\frac{2\gamma}{n-1}Yf_{\nu} \\
&+\Big(1-\frac{1}{n-1}\Big)X^2+\Big(\frac{2\gamma}{n-1}-2\gamma\Big)XY
+\Big(\frac{n-2}{n-1}\gamma^2-\gamma\Big)Y^2\bigg]u^{-\gamma}e^{-f}.
\end{split}
\]
Combining two formulas above gives that
\[
\begin{split}
 A''(0){-}G'(v_0)V''(0)
&\le\int_{\p E}\bigg[-(n-1)\lambda-\frac{f_{\nu}^2}{(n-1)}
+\Big(k{-}\frac{2}{n-1}\Big)Xf_{\nu}+\frac{2\gamma}{n-1}Yf_{\nu}\\
&\qquad\qquad-\frac{X^2}{n-1}+\Big(\frac{2\gamma}{n-1}-\alpha\Big)XY
+\Big(\frac{n-2}{n-1}\gamma^2-\gamma\Big)Y^2\bigg]u^{-\gamma}e^{-f}.
\end{split}
\]
Since $k=\frac{2}{n-1}$ and $\alpha=\frac{2\gamma}{n-1}$, then
\[
\begin{split}
&A''(0)-G'(v_0)V''(0)\\
&\le\int_{\p E}\bigg[-(n-1)\lambda-\frac{X^2}{n-1}-\frac{f_{\nu}^2}{n-1}
+\frac{2\gamma}{n-1}Yf_{\nu}+\Big(\frac{n-2}{n-1}\gamma^2-\gamma\Big)Y^2\bigg]u^{-\gamma}e^{-f}.
\end{split}
\]
Applying $-\frac{1}{n-1}f_{\nu}^2+\frac{2\gamma}{n-1}Yf_{\nu}\le\frac{\gamma^2}{n-1}Y^2$
and $0\le\gamma\le 1$, we further get
\begin{equation*}
\begin{split}
A''(0)-G'(v_0)V''(0)&\le\int_{\p E}\bigg[-(n-1)\lambda-\frac{X^2}{n-1}
+\Big(\gamma^2-\gamma\Big)Y^2\bigg]u^{-\gamma}e^{-f}\\
&\le\int_{\p E}\Big[-(n-1)\lambda-\frac{X^2}{n-1}\Big]u^{-\gamma}e^{-f}.
\end{split}
\end{equation*}
Combining this with $V'(0)=\int_{\de E}u^{\alpha-\gamma} e^{-(k+1)f}$
and $X=G'(v_{0})u^{\alpha-\gamma}e^{-kf}$, we obtain that
\[
\begin{split}
G''(v_0)= &\frac{A''(0)-G'(v_0)V''(0)}{V'^2(0)}\\
\le &-(n-1)\lambda\cdot\frac{\int_{\p E}u^{-\gamma}e^{-f}}{\left(\int_{\p E}u^{\alpha-\gamma}e^{-(k+1)f}\right)^2}
-\frac{G'^2(v_0)}{n-1}\cdot\frac{\int_{\p E}u^{2\alpha-3\gamma}e^{-(2k+1)f}}{\left(\int_{\p E}u^{\alpha-\gamma}e^{-(k+1)f}\right)^2}.
\end{split}
\]
Since $\inf_Mu=1$ and $\gamma\ge\alpha$, we observe that
\[
\begin{split}
\int_{\p E}u^{-\gamma}e^{-f}
&=\int_{\p E}u^{2\gamma-2\alpha}e^{2kf} \cdot u^{2\alpha-3\gamma}e^{-(2k+1)f} \\
&\ge e^{2kf_{min}}\int_{\p E}u^{2\alpha-3\gamma}e^{-(2k+1)f}.
\end{split}
\]
Hence,
\[
G''(v_0) \leq \left[-(n-1)\lambda e^{2kf_{min}}
-\frac{G'^2(v_0)}{n-1}\right]\cdot\frac{\int_{\p E}u^{2\alpha-3\gamma}e^{-(2k+1)f}}{\left(\int_{\p E}u^{\alpha-\gamma}e^{-(k+1)f}\right)^2}.
\]
By the Cauchy-Schwarz inequality, we see that
\[
\begin{split}
\left(\int_{\p E}u^{\alpha-\gamma}e^{-(k+1)f}\right)^2
=& \left(\int_{\p E}u^{\alpha-\frac{3}{2}\gamma}e^{-k-\frac{1}{2}f} \cdot u^{\frac{1}{2}\gamma}e^{-\frac{1}{2}f}\right)^2\\
\le & \left(\int_{\p E}u^{2\alpha-3\gamma}e^{-(2k+1)f}\right)\left(\int_{\p E}u^{\gamma}e^{-f}\right),
\end{split}
\]
So
\[
G''(v_0)\le\left[-(n-1)\lambda e^{2kf_{min}}-\frac{G'^2(v_0)}{n-1}\right]
\cdot\frac{1}{\int_{\p E}u^{\gamma}e^{-f}}.
\]
Since
\[
G(v_0)=I(v_0)=\int_{\p E}u^{\gamma}e^{-f},
\]
then \eqref{goal G} follows.
\end{proof}

Finally we shall apply the preceding lemmas to prove Theorem \ref{result}(2).
\begin{proof}[Proof of (2) in Theorem \ref{result}]

We first prove the conclusion when $3\le n\le 7$. Combining Lemma \ref{ODE comparison},
\eqref{ODE comparison condition a}, Lemma \ref{limestiI} and Lemma
\ref{ODE comparison condition c}, we see that $J:=e^{-\frac{f_{max}}{n}}I$ is a continuous
function on $[0,V_1)$ and it satisfies conditions (1)-(3) in Lemma \ref{ODE comparison}
with $V_1=e^{-\frac{f_{max}}{n}}V_0$ and $\Lambda=\lambda e^{2kf_{min}-\frac{2}{n}f_{max}}$.
Moreover, the conclusion of Lemma \ref{limestiI} gives
\[
e^{-\frac{f_{max}}{n}}V_0\le\left(\lambda e^{2kf_{min}-\frac{2}{n}f_{max}}\right)^{-\frac n2}|\S^n|,
\]
which implies that
\[
V_0\le e^{(1+\frac 1n)f_{max}-knf_{min}}\lambda^{-\frac n2}|\S^n|.
\]
Since $\inf_M u=1$ and $\alpha\ge 0$, then
\[
V_0=\int_Mu^{\alpha}e^{-(k+1)f}\ge e^{-kf_{max}}\int_{M}e^{-f}=e^{-kf_{max}}|M|_f.
\]
Combining this and using $k=\frac{2}{n-1}$, the conclusion follows
when $3\le n\le7$.

When $n\ge8$, $\de E$ may have singularities. The global volume comparison is
proved by modifying the argument of the case $3\le n\le 7$ as Appendix A in \cite{AX17}.
\end{proof}


\section{Weighted spectral splitting theorem}\label{split}

In this section, we prove Theorem \ref{result2} in introduction by
following the arguments of \cite{APX24,Yeu}. To achieve this aim,
we need several technical lemmas, which will be given below.

Let $(M,g,e^{-f}dv)$ be a complete noncompact weighted manifold with
$|f|\le\theta$ for some constant $\theta\ge 0$ on $M$. For $q\in C^{\infty}(M)$
and a bounded domain $D\subset M$, let
\[
\lambda_{1,f}(D)<\lambda_{2,f}(D)\le\lambda_{3,f}(D)\le \cdots
\]
be the sequence of eigenvalues of $\Delta_f-q$ acting on functions vanishing
on $\partial D$. The variational characterization of $\lambda_{1,f}(D)$ is given by
\[
\lambda_{1,f}(D)=\inf\left\{\int_D\left(|\D u|^2+qu^2\right)e^{-f}dv \Big|\,\,
\mathrm{supp}\, u\subset D,\,  \int_D u^2 e^{-f}dv=1\right\}.
\]
Using this characterization, we have $\lambda_{1,f}(D_1)\ge\lambda_{1,f}(D_2)$
for $D_1\subset D_2$, where $D_i$ ($i=1,2$) are connected domains in $M$.
Furthermore, if $D_2\setminus D_1\neq \emptyset$, then $\lambda_{1,f}(D_1)>\lambda_{1,f}(D_2)$.
Besides, we have some standard 2nd order elliptic equation theories for
the weighted schrodinger equation $(\Delta_f-q)h=0$ in $(M,g,e^{-f}dv)$,
such as the strong maximum principle, the regularity theorem, the Harnack
inequality, the Schauder interior estimate, etc. (see Sections 6 and 8 in
\cite{GT98}). Therefore, we have the following equivalent statement
proved in Theorem 2.3 of \cite{Yeu}, slightly generalizing Theorem 1 in
\cite{FS80}.
\begin{lemma}\label{equvlem}
The following conditions are equivalent:
\begin{itemize}
\item [(1)]$\lambda_{1,f}(D)\ge 0$ for every bounded domain $D\subset M$,
\item [(2)]$\lambda_{1,f}(D)> 0$ for every bounded domain $D\subset M$,
\item[(3)]There exists a positive function $h$ satisfying $(\Delta_f-q)h=0$ on $M$.
\end{itemize}
\end{lemma}

We also have the following constructional result appeared in Lemma 2.1
of \cite{APX24}, which will be used in the proof of Theorem \ref{result2}.
\begin{lemma}\label{bubfun}
For any $\e\in(0,\frac12)$, there exists a smooth function
$\bar h_{\e}:(-\frac{1}{\e},\frac{1}{\e})\to\mathbb{R}$ such that
\begin{enumerate}
\item $\bar h_{\e}'+\bar h_{\e}^2\ge{\e}^2$ on
$(-\frac{1}{\e},-1]\cup[1,\frac{1}{\e})$,

\item $|\bar h_{\e}'+\bar h_{\e}^2|\le C\e$ for a universal constant
$C>0$ on $[-1,1]$,

\item $(\bar h_{\e})'<0$ on $(-\frac{1}{\e},\frac{1}{\e})$,
$\bar h_\e(0)=0$ and $\lim_{x\to\pm\frac{1}{\e}} \bar h_{\e}(x)=\mp\infty$,

\item $\bar h_{\e}\to0$ smoothly as $\e\to0$ on any compact subset of $\R$.
\end{enumerate}
\end{lemma}

By Lemma \ref{equvlem}, we see that condition \eqref{speccond} of Theorem
\ref{result2} is equivalent to a fact that there exist $\alpha\in (0,1)$
and a positive function $u\in C^{2,\alpha}(M)$ such that $-\gamma\Delta_f u+u\,\Ric_f=0$.
For such $u$, we have the following perturbation result, which will
be used in constructing appropriate functions in the proof of Theorem
\ref{result2}.

\begin{lemma}\label{perturb}
Let $(M,g, e^{-f}dv)$ be an $n$-dimensional complete noncompact weighted manifold
with bounded $f$. Assume there exist a constant $\alpha\in (0,1)$ and
$0<u_0\in C^{2,\alpha}(M)$ such that $-\gamma\Delta_f u_0+q u_0=0$, where
$\gamma>0$ and $q\in\mathrm{Lip}_{\mathrm{loc}}(M)$. Then for any point $x\in M$
and $r\in(0,1)$, there exists $w\in C^{2,\alpha}(M)$ such that
\begin{enumerate}
\item $-2u_0\le w<0$ on $M$,
\item $-\gamma\Delta_f w+q w>0$ on $M\setminus B(x,r)$.
\end{enumerate}
\end{lemma}

\begin{proof}[Proof of Lemma \ref{perturb}]
The proof follows by the argument of Lemma 2.2 of \cite{APX24} by
changing the Riemannian measure to the weighted measure. The
detailed discussion can be referred to Lemma 2.5 in \cite{Yeu}.
\end{proof}

Next, we have auxiliary density estimates for the unequally warped
$\mu$-bubble in weighted manifolds, which will be also used in the
proof of Theorem \ref{result2}. Notice that Yeung (see Lemma 2.6 in \cite{Yeu})
obtained the same conclusions of density estimates for a slightly
different functional.
\begin{lemma}\label{densest}
Let $(M,g, e^{-f}dv)$ be an $n$-dimensional complete noncompact weighted manifold
with $|f|\le\theta$ for some constant $\theta\ge0$. Let $\phi:M\to\R$ be a
surjective smooth proper $1$-Lipschitz function satisfying
$\Omega_0:=\phi^{-1}\left((-\infty,0)\right)$ has smooth bounded boundary.
For any $\e\in(0,\frac12)$, set
\[
h_\e:=c\,\bar h_\e\circ\phi,
\]
where $c>0$ is a constant and $\bar h_\e$ is given by Lemma \ref{bubfun}.
For $\gamma\ge 0$ and $0<u\in C^{2,\alpha}(M)$ for some $\alpha\in (0,1)$,
define the functional
\[
\mathcal P_{f,\e}(E):=\int_{\p^*E} u^\gamma e^{-f}
-\int_M\big(\chi_E-\chi_{\Omega_0}\big)h_\e u^\gamma e^{-(1+\frac{2}{n-1})f}
\]
for any \textit{admissible} set of locally finite perimeter $E\subset M$, i.e.,
such that $E\Delta\Omega_0\Subset\phi^{-1}((-\frac{1}{\e}, \frac{1}{\e}))$
up to negligible sets.

Assume $\Omega_\e$ is a minimizer for $\mathcal P_{f,\e}$ with respect to
admissible compactly supported variations. Then, up to modify $\Omega_\e$
on a negligible set, the following conclusions hold:
\begin{enumerate}
\item [(1)] $\partial^* \Omega_\e$ is an hypersurface of class $C^{2,\alpha'}$
for some $\alpha'\in(0,1)$, and $\partial\Omega_\e\setminus \partial^* \Omega_\e$
is a closed set with Hausdorff dimension no more than $n-8$.
\item [(2)]
Let $A\subset M$ be a precompact open set and $\e_0\in(0,\frac12)$ such that
$\overline{A}\subset \phi^{-1}((-\frac{1}{\e_0}, \frac{1}{\e_0}))$. Then
there exist $r_0$ and $C_d\in(0,1)$ depending on $n$, $\gamma$,
$\|\log u\|_{L^\infty(A)}$, $g|_{\overline{A}}$, $\e_0$, $\theta$ and $c$
such that
\[
C_d \le \frac{|\p\Omega_\e \cap B(y,\rho)|}{\rho^{n-1}} \le C_d^{-1}
\]
 for any $\e<\e_0$ and for any $y \in A \cap \p\Omega_\e$ with $B(y,2r_0)\Subset A$, and any $\rho<r_0$.
\end{enumerate}
\end{lemma}

\begin{proof}[Proof of Lemma \ref{densest}]
The conclusion (1) exactly follows by the Riemannian case; see Theorems 27.5 and 28.1
in \cite{Maggi12}. The conclusion (2) follows by the argument of \cite{APX24,Yeu}.
Since the functional $\mathcal P_{f,\e}(E)$ in Lemma \ref{densest} is
different from the case of \cite{Yeu}, we retain our complete proof of conclusion (2)
here.

In the following, we let $C>0$ be a constant depending on $n$, $\gamma$, $q$,
$||\log u||_{L^\infty(A)}$, $g|_{\overline{A}}$, $\e_0$, $c$ and $\theta$, which may
change from line to line. Let $y\in A$ and $\rho>0$ such that $B(y,\rho)\subset A$.
Since $h_\e\to 0$ on $\overline{A}$ as $\e\to 0$, then
$||h_\e||_{L^\infty(\overline{A})}\le C$ for any $\e<\e_0$. So, for any arbitrary
set $F$ of locally finite perimeter such that $F\Delta \O_\e\subset B(y,\rho)$,
since $|f|\le\theta$ for some constant $\theta\ge 0$, we have
\begin{equation}
\begin{split}\label{estupp}
\int_{\p^*\O_\e}u^\gamma e^{-f}&\le \int_{\p^* F}u^\gamma e^{-f}
+\int h_\e u^{\gamma}(\chi_{\O_\e}-\chi_F) e^{-(1+\frac{2}{n-1})f}\\
&\le\int_{\p^* F}u^\gamma e^{-f}+C\,||h_\e||_{L^\infty(A)}\cdot||u||^\gamma_{L^\infty(A)}
\cdot|F\Delta\O_\e|\\
 &\le\int_{\p^* F}u^\gamma e^{-f}+C|F\Delta \O_\e|.
\end{split}
\end{equation}
We also notice that
\begin{equation}
\begin{split}\label{estul}
\int_{\p^*\O_\e\setminus \p^*F}u^\gamma e^{-f}
&\ge\int_{(\p^*\O_\e\setminus \p^*F)\cap B(y,\rho)}u^\gamma e^{-f}\\
&\ge\int_{\p^*\O_\e\cap B(y,\rho)}u^\gamma e^{-f}-\int_{\p^*F\cap B(y,\rho)}u^\gamma e^{-f}.
\end{split}
\end{equation}
Combining \eqref{estupp} and \eqref{estul}, by the upper and lower bound
of $u$ in $\overline{A}$, we get
\[
\begin{split}
C|\p^*\O_\e\cap B(y,\rho)|&\le\int_{\p^*\O_\e\cap B(y,\rho)}u^\gamma e^{-f}\\
&\le \int_{\p^*F\cap B(y,\rho)}u^\gamma e^{-f}+\int_{\p^*\O_\e\setminus \p^*F}u^\gamma e^{-f}\\
&\le \int_{\p^*F\cap B(y,\rho)}u^\gamma e^{-f}+C|F\Delta \O_\e|\\
&\le C\left(|\p^*F\cap B(y,\rho)|+|F\Delta \O_\e|\right).
\end{split}
\]
Since the local isoperimetric inequality with Euclidean exponents for sets
contained in $A$ and the Ahlofs-type bound contained in $A$ are still hold
in weighted measure, using $|f|\le \theta$ for some constant $\theta\ge 0$,
then we have
\begin{equation}
\begin{split}\label{isopest}
|\p^*\Omega_\e \cap B(y,\rho)|
&\le C\left[|\p^*F\cap B(y,\rho)|
+|F\Delta \Omega_\e|^{\frac1n}|F\Delta\Omega_\e|^{1-\frac 1n}\right]\\
&\le C\left[|\p^*F\cap B(y,\rho)|+|B(y,\rho)|^{\frac1n}\cdot C_{\rm iso}|\p^*(F\Delta\Omega_\e)|\right]\\
&\le C\left[|\p^*F\cap B(y,\rho)|+C\rho|\p^*(F\Delta\Omega_\e)|\right]\\
&\le C\left[|\p^*F\cap B(y,\rho)|+C\rho\Big(|\p^*\Omega_\e\cap B(y,\rho)|+|\p^*F\cap B(y,\rho)|\Big)\right].
\end{split}
\end{equation}
If $\rho$ is small enough, then
\[
|\p^*\Omega_\e \cap B(y,\rho)|\le C |\p^*F\cap B(y,\rho)|.
\]
So, $\Omega_\e$ is a $C$-quasiminimal set in $B(y,2r_0)\Subset A$ in the
sense of Definition 3.1 in \cite{Kin}. The desired density estimates follow
from Lemma 5.1 in \cite{Kin} and the conclusion (1).
\end{proof}

Using the above lemmas, we can follow the argument of \cite{APX24} to prove Theorem
\ref{result2}. Since the proof is very similar to the manifold case in \cite{APX24},
here we only provide some brief proof steps. Some technique details we omitted
are referred to \cite{APX24}.

\begin{proof}[Proof of Theorem \ref{result2}]
If $\gamma\le 0$, this case returns to the setting: $\Ric_f\ge0$ point-wise
and $f$ is bounded in $(M,g,e^{-f}dv)$. So the conclusion follows by the
Cheeger-Gromoll type splitting theorem for the $\infty$-Bakry-Emery Ricci
curvature due to Lichnerowicz \cite{Lic} (see also Theorem 6.1 in \cite{WW09}).
In the rest, we may assume $\gamma>0$. By the preceding statement, since $\lambda_1(-\gamma\Delta_f+\Ric_f)\ge0$, then there exist $\alpha\in(0,1)$ and
$0<u_0\in C^{2,\alpha}(M)$ such that $-\gamma\Delta_f u_0+\Ric_f\cdot u_0=0$.
In the following, we shall \emph{claim} that
\[
|\D u_0|(x)=0
\]
for any point $x\in M$. If the claim is true, then function $u_0$ is constant on
$M$ and hence $\Ric_f\ge0$ point-wise on $M$. Combining this with the bounded
function $f$, the theorem follows by Lichnerowicz's splitting result \cite{Lic}.

Below we shall prove the claim $|\D u_0|(x)=0$ for any point $x\in M$. Since
$M$ has more than one end, we can follow the construction of \cite{APX24} to
obtain a smooth surjective proper function $\phi:M\to\R$ such that $|\D\phi|\le 1$,
$|\phi(x)|<\frac12$ and $\Omega_0:=\phi^{-1}((-\infty,0))$ has smooth boundary.
Let $0<\delta,r\ll1$ satisfy
\begin{equation}\label{eq:radii_constraints}
B(x,1/\delta)\Supset\phi^{-1}([-2,2]),\quad
B(x,2r)\Subset\phi^{-1}\big([-1,1]\big).
\end{equation}
We use Lemma \ref{perturb} to ball $B(x,r)$ with $u_0$
and $q=\Ric_f$. We let $w_r$ denote the resulting function.
For a parameter $a>0$, we set $u_{r,a}:=u_0+aw_r$. There exists
$a_0=a_0(\delta,w_r,u_0)$ so that for all $a<a_0$,
\begin{equation}\label{eq:choice_a_1}
u_0/2\le u_{r,a}\le u_0,\quad ||u_{r,a}-u_0||_{C^2(B(x,1/\delta))}<\delta,
\end{equation}
and
\begin{equation}\label{eq:choice_a_2}
\left\{\begin{aligned}
&-\gamma\Delta_f u_{r,a}+u_{r,a}\,\Ric_f>0\quad\text{in $M\setminus B(x,r)$,} \\
&-\gamma\Delta_f u_{r,a}+u_{r,a}\,\Ric_f\ge-\delta \cdot u_{r,a}
\quad\text{in $B(x,r)$.}
\end{aligned}\right.
\end{equation}
Particularly, by \eqref{eq:radii_constraints} and \eqref{eq:choice_a_2},
there exists a constant $\mu=\mu(a,w_r,u_0)>0$ such that
\begin{equation}\label{eq:choice_a_3}
-\gamma\Delta_f u_{r,a}+u_{r,a}\,\Ric_f\ge\mu\cdot u_{r,a}
\end{equation}
in $\phi^{-1}([-2,2])\setminus B(x,2r)$.

Since $|f|\le\theta$ for some constant $\theta\ge 0$ on $M$,
for a sufficiently small constant $c_0=c_0(\theta, n,\gamma)>0$,
which will be determined below in \eqref{defc_0}, and for any
$\e<\frac{1}{10}$, function $h_\e$ is defined by
\[
h_\e(x):=c_0^{-1}\bar h_\e(\phi(x))
\]
for all $x\in\phi^{-1}((-\tfrac{1}{\e},\tfrac{1}{\e}))$,
where $\bar h_\e$ is given by Lemma \ref{bubfun}. In particular,
we have
\begin{equation}\label{eq:dh+h2}
c_0h_\e^2-|\nabla h_\e|
= c_0^{-1}\left[\bar h_\e(\phi)^2-|\bar h'_\e(\phi)|\cdot|\nabla\phi|\right]
\ge c_0^{-1}\e^2
\end{equation}
on $\phi^{-1}\big((-\frac{1}{\e},-1]\cup[1,\frac{1}{\e})\big)$,
where we used that $\phi$ is $1$-Lipschitz and Lemma \ref{bubfun}(1)(3).
By Lemma \ref{bubfun}(2)(3),  we have
\begin{equation}\label{eq:dh+h2f}
c_0h_\e^2-|\D h_\e|\ge-Cc_0^{-1}\e
\end{equation}
inside $\phi^{-1}\big([-1,1]\big)$, where $C$ is the constant in Lemma \ref{bubfun}(2). For $\delta$, $a$ and
$r$ small enough as indicated above, we can perform the $\mu$-bubble argument
as in \cite{APX24}. Below we let $u=u_{r,a}$ for notational simplicity.
For any set of finite perimeter $\Omega$ with
$\Omega\Delta\Omega_0\subset\phi^{-1}((-\frac{1}{\e},\frac{1}{\e}))$,
we define the following energy functional
\[
\mathcal{P}_{f,\e}(\Omega):=\int_{\p^*\Omega}u^\gamma e^{-f}
-\int_M(\chi_\Omega-\chi_{\Omega_0}) h_{\e} u^\gamma e^{-(1+\frac{2}{n-1})f}.
\]
As discussion in Proposition 12 in \cite{CL24}, since $h_\e$ diverges as
$\phi\to\pm\frac{1}{\e}$ by Lemma \ref{bubfun}(3), there exists
a minimizer $\Omega_\e$ for $\mathcal{P}_{f,\e}$ with
$\Omega_\e\Delta\Omega_0\Subset\phi^{-1}((-\frac{1}{\e},\frac{1}{\e}))$.

Since $u\in C^{2,\alpha}(M)$, we can apply the regularity property in Lemma \ref{densest}(1)
to $\Omega_\e$. For any $\vp\in C^{\infty}(M)$, let $\{F_t\}_{t\in(-t_0,t_0)}$
be a smooth family of sets with $F_0=\Omega_\e$ and $\frac{\p F_t}{\de t}=\vp\nu$
at $t=0$, where $\nu$ is the outer unit normal of $\p^*\Omega_\e$.
From the argument of \cite{APX24}, here we only need to consider the case $3\le n\le 7$.
Because when $n\ge 8$, the boundary $\p^*\Omega_\e$ may contain singularities.
But we can apply the argument of Appendix A in \cite{AX17} to carry
out our the following computations.

The first variation of $\mathcal{P}_{f,\e}(F_t)$ is that
\[
0 = \frac{\mathrm{d}\mathcal{P}_{f,\e}(F_t)}{\mathrm{d}t}\Bigg|_{t=0}
= \int_{\p^*\Omega_{\e}}\left(H_f+\gamma u^{-1}u_\nu-h_{\e}e^{-\frac{2f}{n-1}}\right)u^\gamma\vp e^{-f}.
\]
Since $\vp$ is arbitrary, we have
\[
H_f=h_\e e^{-\frac{2f}{n-1}}-\gamma u^{-1}u_\nu.
\]
Then we compute its second variation
\[
\begin{split}
0\le&\frac{\mathrm{d}^2 \mathcal{P}_{f,\e}(F_t)}{\mathrm{d}t^2}\Bigg|_{t=0}\\
= &\int_{\p^*\Omega_{\e}}
\Big[-\Delta_{\p^*\Omega_{\e}}\vp-|\mathrm{II}|^2\vp-\Ric_f(\nu,\nu)\vp-\gamma u^{-2}u_\nu^2\vp\\
&\qquad\quad+\gamma u^{-1}\vp\left(\Delta_fu-\Delta_{\p^*\Omega_{\e}}u
+\langle\D u,\nabla f\rangle-Hu_\nu\right)
-\gamma u^{-1}\langle\D_{\p^*\Omega_{\e}}u,\D_{\p^*\Omega_{\e}}\vp\rangle\\
&\qquad\quad+\langle\D_{\p^*\Omega_{\e}}f,\D_{\p^*\Omega_{\e}}\vp\rangle
-\langle\D h_{\e},\nu\rangle e^{-\frac{2f}{n-1}}\vp
+\frac{2}{n-1}h_{\e}e^{-\frac{2f}{n-1}}f_{\nu}\vp\Big] u^\gamma\vp e^{-f}.
\end{split}
\]
Since $\p\Omega_{\e}$ is compact, taking $\vp:= u^{-\gamma/2}$, then
\begin{equation}
\begin{split}\label{PQResti}
0\le&\int_{\p^*\Omega_{\e}}\Big[-u^{\frac{\gamma}{2}}\Delta_{\p^*\Omega_{\e}}(u^{-\frac{\gamma}{2}})
-\gamma u^{-1}\Delta_{\p^*\Omega_{\e}}u-\gamma u^{\frac{\gamma}{2}-1}
\langle\D_{\p^*\Omega_{\e}}u,\D_{\p^*\Omega_{\e}}(u^{-\frac{\gamma}{2}})\rangle\Big]e^{-f}\\
&+\int_{\p^*\Omega_{\e}}\Big[-|\mathrm{II}|^2-\gamma u^{-2}u_\nu^2-\gamma H u^{-1}u_\nu
+\gamma u^{-1}\langle\D u,\nabla f\rangle-\langle\D h_{\e},\nu\rangle e^{-\frac{2f}{n-1}}\Big]e^{-f}\\
&+\int_{\p^*\Omega_{\e}}\Big[\gamma u^{-1}\Delta_fu-\Ric_{f}(\nu,\nu)
+u^{\frac{\gamma}{2}}\langle\D_{\p^*\Omega_{\e}} f,\D_{\p^*\Omega_{\e}}(u^{-\frac{\gamma}{2}})\rangle
+\frac{2 h_{\e}}{n-1}e^{-\frac{2f}{n-1}}f_{\nu}\Big]e^{-f}\\
=&\int_{\p^*\Omega_{\e}}\Big[-u^{\frac{\gamma}{2}}{\Delta}_{\p^*\Omega_{\e}}(u^{-\frac{\gamma}{2}})
-\gamma u^{-1}{\Delta}_{\p^*\Omega_{\e}}u
-\gamma u^{\frac{\gamma}{2}-1}\langle\D_{\p^*\Omega_{\e}}u,\D_{\p^*\Omega_{\e}}(u^{-\frac{\gamma}{2}})\rangle\\
&\qquad\quad+\frac{\gamma}{2}u^{-1}\langle\D_{\p^*\Omega_{\e}} f,\D_{\p^*\Omega_{\e}}u\rangle\Big]e^{-f}\\
&+\int_{\p^*\Omega_{\e}}\Big[-|\mathrm{II}|^2-\gamma u^{-2}u_\nu^2-\gamma H_f u^{-1}u_\nu
-\langle\D h_{\e},\nu\rangle e^{-\frac{2f}{n-1}}+\frac{2 h_{\e}}{n-1}e^{-\frac{2f}{n-1}}f_{\nu}\Big]e^{-f}\\
&+\int_{\p^*\Omega_{\e}}\Big[\gamma u^{-1}\Delta_fu-\Ric_{f}(\nu,\nu)\Big]e^{-f}\\
=:&\int_{\p^*\Omega_\e} (P+Q+R)e^{-f},
\end{split}
\end{equation}
where in the second equality, we used the following relation
\[
\gamma u^{-1}\langle\D u,\nabla f\rangle
+u^{\frac{\gamma}{2}}\langle\D_{\p^*\Omega_{\e}}f,\D_{\p^*\Omega_{\e}}(u^{-\frac{\gamma}{2}})\rangle
=\gamma u^{-1}u_{\nu}f_{\nu}
+\frac{\gamma}{2}u^{-1}\langle\D_{\p^*\Omega_{\e}} f,\D_{\p^*\Omega_{\e}}u\rangle.
\]
In the above inequality, three integral terms on the right hand side can be
further estimated by using the same arguments of \cite{APX24}. The detailed
discussions are as follows.

For the first integral term $\int_{\p^*\Omega_\e}Pe^{-f}$, integrating by parts
with respect to the weighted area measure, we have
\begin{equation}
\begin{split}\label{eq:P}
\int_{\p^*\Omega_\e}Pe^{-f}
&=\int_{\p^*\Omega_\e}
\frac{\gamma^2}{4}u^{-2}|\D_{\p^*\Omega_\e}u|^2 e^{-f}-\gamma u^{-2}|\D_{\p^*\Omega_\e}u|^2 e^{-f}\\
&=\int_{\p^*\Omega_\e}
\left(\frac{\gamma^2}{4}-\gamma\right)u^{-2}|\D_{\p^*\Omega_\e}u|^2 e^{-f}.
\end{split}
\end{equation}
Notice that $\frac{\gamma^2}{4}-\gamma<0$ for $0<\gamma<4$.

Next, we estimate the second integral term $\int_{\p^*\Omega_\e}Qe^{-f}$.
Set
\[
X=h_\e e^{-\frac{2f}{n-1}}\quad \mathrm{and} \quad
Y=u^{-1}u_\nu.
\]
Then
\[
H_f=X-\gamma Y.
\]
Combining this and the trace inequality $|\mathrm{II}|^2\ge\frac{H^2}{n-1}$,
we see that
\begin{equation}
\begin{split}\label{etiQ}
Q&=-|\mathrm{II}|^2-\gamma u^{-2}u_\nu^2-\gamma H_f u^{-1}u_\nu
-\langle\D h_{\e},\nu\rangle e^{-\frac{2f}{n-1}}+\frac{2 h_{\e}}{n-1}e^{-\frac{2f}{n-1}}f_{\nu}\\
&\le-\frac{(X-\gamma Y+f_\nu)^2}{n-1}-\gamma Y^2-\gamma Y(X-\gamma Y)
+|\D h_{\e}|e^{-\frac{2f}{n-1}}+\frac{2}{n-1}Xf_{\nu}\\
&=-\frac{X^2}{n-1}+\left(\frac{2\gamma}{n-1}-\gamma\right)XY
+\left(\frac{n-2}{n-1}\gamma^2-\gamma\right)Y^2\\
&\quad-\frac{1}{n-1}f^2_\nu+\frac{2\gamma}{n-1}Yf_\nu+|\D h_{\e}|e^{-\frac{2f}{n-1}}.
\end{split}
\end{equation}

To go on estimating the above $Q$, we need the following two claims.

\textbf{Claim 1:} For any $0<\gamma<\frac{4}{n-1}$, there exist a positive
constant $\e_0=\e_0(n,\gamma)\le\frac{1}{2(n-1)}$ and a positive constant
$\beta_0=\beta_0(n,\gamma):=\gamma-\frac{n-1}{4}\gamma^2-(n-3)^2\gamma^2\e_0$,
such that
\[
-\frac{X^2}{n-1}+\left(\frac{2\gamma}{n-1}-\gamma\right)XY
+\left(\frac{n-2}{n-1}\gamma^2-\gamma\right)Y^2\le -\e_0X^2-\beta_0Y^2.
\]
\begin{remark}\label{este_01}
For example, we can take
\[
\e_0(n,\gamma)=\min\left\{\frac{1}{2(n-1)},\,\,
\frac{1}{(n-3)^2+1}\left(\frac{1}{\gamma}-\frac{n-1}{4}\right)\right\}
\]
to ensure $\beta_0>0$.
\end{remark}

\begin{proof}[Proof of Claim 1]
We first check $\beta_0>0$. If we choose $\e_0=\e_0(n,\gamma)$ as in Remark \ref{este_01},
then
\[
0<\gamma\le\frac{1}{\frac{n-1}{4}+\left[(n-3)^2+1\right]\e_0}.
\]
Using the above estimate, according to the definition of $\beta_0$,
we directly check that $\beta_0>0$. In the rest, we only need to check
that the following $(X,Y)$-polynomial
\[
p(X,Y):=\left(\e_0-\frac{1}{n-1}\right)X^2+\left(\frac{2\gamma}{n-1}-\gamma\right)XY
+\left(\beta_0+\frac{n-2}{n-1}\gamma^2-\gamma\right)Y^2
\]
is non-positive. Noticing that $\e_0-\frac{1}{n-1}<0$ and its discriminant
\[
\begin{split}
\Delta&=\left(\frac{2\gamma}{n-1}-\gamma\right)^2-4\left(\e_0-\frac{1}{n-1}\right)
\left(\beta_0+\frac{n-2}{n-1}\gamma^2-\gamma\right)\\
&=\frac{2(n-3)^2}{n-1}\gamma^2\e_0\Big(2(n-1)\e_0-1\Big)\\
&\le0,
\end{split}
\]
where we used $\e_0\le\frac{1}{2(n-1)}$ in the last inequality, hence
$p(X,Y)\le0$.
\end{proof}

Next we consider the following $(f_\nu,Y)$-polynomial
\[
p(f_\nu,Y):=-\frac{1}{n-1}f^2_\nu+\frac{2\gamma}{n-1}Yf_\nu-\beta_0 Y^2,
\]
where $0<\gamma<\frac{4}{n-1}$. Then we have

\textbf{Claim 2:}
For any $0<\gamma<\left(\frac{1}{n-1}+\frac{n-1}{4}\right)^{-1}$, there
exist a positive constant $\e_0=\e_0(n,\gamma)\le\frac{1}{2(n-1)}$ and
another positive constant $\beta_1=\beta_1(n,\gamma):=\beta_0-\frac{\gamma^2}{n-1}$,
where $\beta_0=\gamma-\frac{n-1}{4}\gamma^2-(n-3)^2\gamma^2\e_0$, such that
\[
p(f_\nu,Y)\le-\beta_1Y^2.
\]

\begin{remark}\label{este_02}
For example, we can take
\[
\e_0(n,\gamma)=\min\left\{\frac{1}{2(n-1)},\,\,
\frac{1}{(n-3)^2+1}\left(\frac{1}{\gamma}-\frac{1}{n-1}-\frac{n-1}{4}\right)\right\}
\]
to ensure $\beta_1>0$.
\end{remark}

\begin{proof}[Proof of Claim 2]
We see that
\[
\begin{split}
p(f_\nu,Y)&=-\frac{1}{n-1}\left(f_\nu-\gamma Y\right)^2
-\left(\beta_0-\frac{\gamma^2}{n-1}\right)Y^2\\
&\le-\left(\beta_0-\frac{\gamma^2}{n-1}\right)Y^2\\
&:=-\beta_1Y^2.
\end{split}
\]
Now we choose $\e_0=\e_0(n,\gamma)$ as in Remark \ref{este_02} and then
\[
0<\gamma\le\frac{1}{\frac{1}{n-1}+\frac{n-1}{4}+\left[(n-3)^2+1\right]\e_0}.
\]
Using the above estimate and the definition of $\beta_1$,
we directly check that
\[
\beta_1=\beta_0-\frac{\gamma^2}{n-1}>0
\]
and Claim 2 follows.
\end{proof}

Combining Claims 1 and 2 with \eqref{etiQ}, we indeed show that for any
$0<\gamma<\left(\frac{1}{n-1}+\frac{n-1}{4}\right)^{-1}$, there exist a positive
constant $\e_0=\e_0(n,\gamma)$ ( for example, $\e_0=\e_0(n,\gamma)$ can be chosen in Remark
\ref{este_02}), a positive constant $\beta_0=\gamma-\frac{n-1}{4}\gamma^2-(n-3)^2\gamma^2\e_0$
and a positive constant $\beta_1=\beta_0-\frac{\gamma^2}{n-1}$, such that
\eqref{etiQ} can be estimated by
\begin{equation*}
\begin{split}
Q&\le -\e_0X^2-\beta_0Y^2-\frac{1}{n-1}f^2_\nu+\frac{2\gamma}{n-1}Yf_\nu+|\D h_{\e}|e^{-\frac{2f}{n-1}}\\
&\le-\e_0 X^2-\beta_1Y^2+|\D h_\e|e^{-\frac{2f}{n-1}}\\
&=-\e_0h_\e^2e^{-\frac{4f}{n-1}}-\beta_1u^{-2}u_\nu^2+|\D h_\e|e^{-\frac{2f}{n-1}},
\end{split}
\end{equation*}
where we used Claim 1 in the first inequality and Claim 2 in the second inequality.
Since $|f|\le \theta$ for some constant $\theta\ge0$, then we further have
\begin{equation}
\begin{split}\label{eqn:Defnc0}
Q&\le-\e_0h_\e^2e^{\frac{-4\theta}{n-1}}-\beta_1u^{-2}u_\nu^2+|\D h_\e|e^{\frac{2\theta}{n-1}}\\
&\le e^{\frac{2\theta}{n-1}}\Big[-c_0h_\e^2-c_0u^{-2}u_\nu^2+|\D h_\e|\Big]
\end{split}
\end{equation}
for some sufficiently small constant $c_0=c_0(\theta, n,\gamma)>0$. For example, we can take
\begin{equation}\label{defc_0}
c_0=\min\left\{e^{\frac{-6\theta}{n-1}}\e_0,\,\, e^{\frac{-2\theta}{n-1}}\beta_1,\,\,
e^{\frac{-2\theta}{n-1}}\left(\gamma-\frac{\gamma^2}{4}\right)\right\},
\end{equation}
where $\e_0=\e_0(n,\gamma)$ is given in Remark \ref{este_02} and
$\beta_1=\beta_0-\frac{\gamma^2}{n-1}$, and where
$\beta_0=\gamma-\frac{n-1}{4}\gamma^2-(n-3)^2\gamma^2\e_0$.
Clearly, the above choice $c_0$ satisfies
\[
c_0\le\min\left\{e^{\frac{-6\theta}{n-1}}\e_0,\,\, e^{\frac{-2\theta}{n-1}}\beta_1\right\}
\quad \mathrm{and} \quad
e^{\frac{2\theta}{n-1}}c_0\le\gamma-\frac{\gamma^2}{4}.
\]
Here we would like to point out that this $c_0=c_0(\theta, n,\gamma)$
(for example, \eqref{defc_0}) is the constant we are using in \eqref{eq:dh+h2}.

By \eqref{eqn:Defnc0}, we therefore have the estimate of $\int_{\p^*\Omega_\e}Qe^{-f}$:
\begin{equation}\label{eqn:Def2}
\int_{\p^*\Omega_\e}Q e^{-f}\le\int_{\p^*\Omega_\e}
e^{\frac{2\theta}{n-1}}\Big[-c_0h_\e^2-c_0u^{-2}u_\nu^2+|\D h_\e|\Big]e^{-f}.
\end{equation}

For the third integral term $\int_{\p^*\Omega_\e}Re^{-f}$, by \eqref{eq:choice_a_2}
and \eqref{eq:choice_a_3}, we estimate that
\begin{equation}
\begin{aligned}\label{eq:R}
\int_{\p^*\Omega_\e}R e^{-f}
&=\int_{\p^*\Omega_{\e}}\Big[\gamma u^{-1}\Delta_fu-\Ric_{f}(\nu,\nu)\Big]e^{-f}\\
&\le\int_{\p^*\Omega_\e}
\Big[\delta\cdot\chi_{B(x,r)}-\mu\cdot\chi_{\phi^{-1}([-2,2])\setminus B(x,2r)}\Big]e^{-f}.
\end{aligned}
\end{equation}

Substituting \eqref{eq:P}, \eqref{eqn:Def2} and \eqref{eq:R} into \eqref{PQResti},
we have that
\begin{equation}\label{eq:2ndvar_final}
\begin{aligned}
0&\le\int_{\p^*\Omega_\e} (P+Q+R)e^{-f} \\
&\le\int_{\p^*\Omega_\e}
\left(\tfrac{\gamma^2}{4}-\gamma\right)u^{-2}|\D_{\p^*\Omega_\e}u|^2 e^{-f}
+\int_{\p^*\Omega_\e}e^{\frac{2\theta}{n-1}}\Big[-c_0h_\e^2-c_0u^{-2}u_\nu^2+|\D h_\e|\Big]e^{-f}\\
&\quad+\int_{\p^*\Omega_\e}\Big[\delta\cdot\chi_{B(x,r)}-\mu\cdot\chi_{\phi^{-1}([-2,2])\setminus B(x,2r)}\Big]e^{-f}\\
&\le\int_{\p^*\Omega_\e}e^{\frac{2\theta}{n-1}}\Big[\!-c_0u^{-2}|\D_{\p^*\Omega_\e}u|^2-c_0u^{-2}u_\nu^2\Big]e^{-f}
+e^{\frac{2\theta}{n-1}}\Big[\!-c_0h_\e^2+|\D h_\e|\Big]e^{-f}\\
&\quad+\int_{\p^*\Omega_\e}\Big[\delta\cdot\chi_{B(x,r)}-\mu\cdot\chi_{\phi^{-1}([-2,2])\setminus B(x,2r)}\Big]e^{-f}\\
&\le\int_{\p^*\Omega_\e}e^{\frac{2\theta}{n-1}}e^{-\theta}
\Big[-c_0u^{-2}|\D_{\p^*\Omega_\e}u|^2-c_0u^{-2}u_\nu^2\Big]\\
&\quad-e^{\frac{2\theta}{n-1}}e^{-\theta} c_0^{-1}\e^2\big|\p^*\Omega_\e\setminus\phi^{-1}([-1,1])\big|
+e^{\frac{2\theta}{n-1}}e^{\theta}Cc_0^{-1}\e\big|\p^*\Omega_\e\cap\phi^{-1}([-1,1])\big|\\
&\quad+\delta e^{\theta}\big|\p^*\Omega_\e\cap B(x,r)\big|
-\mu e^{-\theta}\big|\p^*\Omega_\e\cap\big(\phi^{-1}([-2,2])\setminus B(x,2r)\big)\big|
\end{aligned}
\end{equation}
for some sufficiently small constant $c_0=c_0(\theta,n,\gamma)>0$, where we
used $\frac{\gamma^2}{4}-\gamma\le-e^{\frac{2\theta}{n-1}}c_0$ from \eqref{defc_0}
in the above third inequality, and where we used \eqref{eq:dh+h2} and \eqref{eq:dh+h2f}
and $|f|\le\theta$ in the last inequality. Here constants $\delta$ and $\mu$
are given in \eqref{eq:choice_a_2} and \eqref{eq:choice_a_3}.

Similar to the argument of \cite{APX24}, we see that
$\p^*\Omega_\e\cap\phi^{-1}([-1,1])\ne\emptyset$.
Indeed, otherwise, by \eqref{eq:2ndvar_final} and $B(x,r)\Subset\phi^{-1}\big([-1,1]\big)$,
we have
\[
0\le-e^{\frac{2\theta}{n-1}}e^{-\theta}c_0^{-1}\e^2\big|\p^*\Omega_\e\big|<0.
\]
We also see that $\p^*\Omega_\e\cap B(x,2r)\ne\emptyset$
whenever $\e<e^{-\frac{2n\theta}{n-1}}C^{-1}c_0\mu$.
Indeed, since $\p^*\Omega_\e\cap\phi^{-1}([-1,1])\ne\emptyset$, then
$|\p^*\Omega_\e\cap\phi^{-1}([-2,2])|>0$. Therefore, if
$\p^*\Omega_\e\cap B(x,2r)=\emptyset$, then \eqref{eq:2ndvar_final} implies
\[
0\le e^{\frac{2\theta}{n-1}}e^{\theta}Cc_0^{-1}\e\big|\p^*\Omega_\e\cap\phi^{-1}([-1,1])\big|
-\mu e^{-\theta}\big|\p^*\Omega_\e\cap\big(\phi^{-1}([-2,2])\big)\big|<0.
\]

Now, for each $m\ge m_0$, where $m_0\in \mathbb N$, we define $\delta_m=r_m=\frac{1}{m}$.
If $m_0$ is chosen large enough, then for every $\delta_m$ and $r_m$, we have
sufficiently small $a_m,\mu_m>0$ such that \eqref{eq:radii_constraints}, \eqref{eq:choice_a_1},
\eqref{eq:choice_a_2} and \eqref{eq:choice_a_3} are met. The perturbed functions
are denoted by $u_m=u_0+a_mw_{r_m}$. We choose
\[
\e_m<\min\left\{\tfrac{1}{m},\,\,e^{-\frac{2n\theta}{n-1}}C^{-1}c_0\mu_m\right\},
\]
and then from what we discussed above, we see that there exists
$y_m\in\p^*\Omega_{\e_m}\cap B(x,2r_m)$.

Notice that $y_m \to x$ as $m\to\infty$. Moreover, by \eqref{eq:choice_a_1},
$\log u_m$ is locally bounded uniformly in $L^\infty$ with respect to $m$,
so the density estimates provided by Lemma \ref{densest}(2) are locally
uniform around $x$ as $m\to \infty$. By contradiction, if we assume
$|\D u_0|(x)>0$, from Lemma \ref{densest}(2) and \eqref{eq:choice_a_1}, we
see that there exist two positive constants $\rho$ and $\eta$ such that
\begin{equation}\label{eqn:DensityLowerBound}
\int_{\p^*\Omega_{\e_m}\cap B(x,\rho)}u_m^{-2}|\D u_m|^2 e^{-f}
\ge\int_{\p^*\Omega_{\e_m}\cap B(y_m,\rho/2)} u_m^{-2}|\D u_m|^2e^{-f}
\ge e^{-\theta}\eta>0
\end{equation}
for any large $m$.

On the other hand, we have a fact that $|\p^*\Omega_{\e_m} \cap \phi^{-1}([-1,1])|$
has an upper bound independent of $m$. Indeed, this fact is obtained by comparing
$\mathcal{P}_{f,\e_m}(\Omega_{\e_m})\le\mathcal{P}_{f,\e_m}(\Omega_0)$, which says
that
\[
\int_{\p^*\Omega_{\e_m}}u_m^\gamma e^{-f}\le\int_{\p\Omega_0}u_m^\gamma e^{-f}
+\int_M\big(\chi_{\Omega_{\e_m}}-\chi_{\Omega_0}\big)h_{\e_m}u_m^\gamma e^{-(1+\frac{2}{n-1})f}
\le\int_{\p\Omega_0}u_m^\gamma e^{-f},
\]
where we used the property of $\overline{h}_{\varepsilon_m}$ in Lemma \ref{bubfun}(3)
and the fact that $u_m$ are uniformly bounded from above and below in $\phi^{-1}([-1,1])$
due to \eqref{eq:choice_a_1}.

Combining \eqref{eqn:DensityLowerBound} and \eqref{eq:2ndvar_final}, and
using the above fact just now, we finally get that
\begin{equation*}
\begin{aligned}
0&<e^{\frac{2\theta}{n-1}}e^{-2\theta}c_0\eta\\
&\le\liminf_{m\to\infty}\int_{\p^*\Omega_{\e_m}\cap B(x,\rho)}
e^{\frac{2\theta}{n-1}}e^{-\theta}c_0u_m^{-2}|\D u_m|^2e^{-f}\\
&\le\liminf_{m\to\infty}\, e^{\frac{2\theta}{n-1}} e^{\theta} Cc_0^{-1}\e_m
\big|\p^*\Omega_{\e_m}\cap\phi^{-1}([-1,1])\big|
+\liminf_{m\to\infty}\, \delta_m  e^{\theta}\big|\p^*\Omega_{\e_m}\cap B(x,r_m)\big|\\
&=0,
\end{aligned}
\end{equation*}
which yields a contradiction. Therefore $|\D u_0|(x)=0$ for any point $x\in M$.
This completes the proof of Theorem \ref{result2}.
\end{proof}

\textbf{Acknowledgements}.
The author thanks Zihang Hao for helpful discussions.
The author also thanks the anonymous referees for making valuable
comments and suggestions which helped to improve the presentation of this work.

\

\textbf{Author Contributions} Jia-Yong Wu is the single author of the manuscript.

\

\textbf{Funding}.  This work is partially supported by National Natural Science Foundation of China
(Grant No. 12571144).

\

\textbf{Data availability} No datasets were generated or analysed during the current study.

\

\begin{center}
\textbf{\large{Declarations}}
\end{center}

\

\textbf{Ethical Approval} Not applicable.

\

\textbf{Conflicts of Interest} The author declares no conflicts of interest.

\

\textbf{Competing interests} The author declares no competing interests.


\end{document}